\documentclass[11pt,twoside]{article}

\usepackage{amsmath}
\usepackage{amssymb}
\usepackage{amsfonts}
\usepackage{mathrsfs}
\usepackage{graphicx}
\usepackage{float}
\usepackage{needspace}
\usepackage{epsfig}
\usepackage{accents}
\usepackage{geometry}
\usepackage[pagewise]{lineno}
\usepackage{amscd,amsmath,amssymb,amsthm,amsfonts,epsfig,graphics}
\usepackage{fancyhdr}
\usepackage{tikz}
\usepackage{pgf,pgfplots}
\usepackage{caption}
\usepackage{titlesec}
\usepackage{setspace}

\usetikzlibrary{calc,arrows.meta}

\usepackage{indentfirst}
\usepackage[skins]{tcolorbox}
\tcbuselibrary{skins, breakable}

\usepackage{abstract}

\usepackage[numbers, sort]{natbib}
\newcommand \footnoteONLYtext[1]
{
\let \mybackup \thefootnote
\let \thefootnote \relax
\footnotetext{#1}
\let \thefootnote \mybackup
\let \mybackup \imareallyundefinedcommand
}

\usepackage{enumitem}
\setenumerate[1]{itemsep=-4pt, partopsep=0pt, parsep=\parskip, topsep=5pt}
\setitemize[1]{itemsep=-4pt, partopsep=0pt, parsep=\parskip, topsep=5pt}
\setdescription{itemsep=-4pt, partopsep=0pt, parsep=\parskip, topsep=5pt}

\allowdisplaybreaks

\newtheorem{definition}{Definition}[section]
\newtheorem{claim}{Claim}
\theoremstyle{plain}
\newtheorem{thm}{Theorem}
\newtheorem{lem}{Lemma}[section]

\newtheorem{fact}{Fact}[section]
\newtheorem{prop}{Proposition}
\newtheorem{remark}{Remark}

\newtheorem*{nonametheorem}{Main Theorem}

\numberwithin{equation}{section}
\numberwithin{figure}{section}
\numberwithin{table}{section}

\newcommand\keywords[1]{\it{Keywords}: #1}
\newcommand\msc[1]{\it{2010 Mathematics Subject Classification}:#1}
\makeatletter
\renewcommand{\section}{\@startsection{section}{1}{0mm}
{-\baselineskip}{0.5\baselineskip}{\normalsize\bf\leftline}}
\makeatother
\makeatletter
\renewcommand{\subsection}{\@startsection{subsection}{1}{0mm}
{-\baselineskip}{0.5\baselineskip}{\normalsize\bf\leftline}}
\makeatother

\tikzset{
	my rectangle/.pic={
		\draw (0,-0.07) rectangle (0.8,0.07);
	}
}
\tikzset{
	my draw/.pic={
		\draw (0,0) -- (4,0);
		\draw (0,-0.1) -- (0,0.1);
		\draw (4,-0.1) -- (4,0.1);
		\draw[dashed] (4,0) -- (5.2,0);
		\draw (1.4,-0.07) rectangle (2.6,0.07);
		\fill (2,0) circle (0.07);
	}
}
\tikzset{
	my square/.pic={
		\draw (-2,-2) rectangle (2,2);
		\draw[dotted] (-2,-2) -- (2,2);
		\draw (-1.5,-2.1) -- (-1.5,-1.9);
		\fill (-0.75,-2)node[shift={(-90:10pt)}] {$c$} circle (1.5pt);
		\draw (0,-2.1)node[shift={(-90:7pt)}] {$p$} -- (0,-1.9);
		\fill (0.75,-2)node[shift={(-90:10pt)}] {$d$} circle (1.5pt);
		\draw (1.5,-2.1) -- (1.5,-1.9);
	}
}
\tikzset{
	my squ/.pic={
		\draw (-1.6,-1.6) rectangle (1.6,1.6);
		\draw[dotted] (-1.6,-1.6) -- (1.6,1.6);
		\draw[dotted] (-1.4,-1.6) -- (-1.4,1.4) -- (1.4,1.4) -- (1.4,-1.6);
		\draw[dotted] (-1.4,-1.4) -- (1.4,-1.4);
		\draw[dotted] (-0.2,1.4) -- (-0.2,-1.6);
		\draw[dotted] (0.2,1.4) -- (0.2,-1.6);
		\draw[dotted] (-1.4,0.2) -- (1.4,0.2);
		\draw[dotted] (-1.4,-0.2) -- (1.4,-0.2);
		\fill (-0.8,-1.6)node[shift={(-90:9pt)}] {$d$} circle (1pt);
		\fill (0.8,-1.6)node[shift={(-90:8pt)}] {$c$} circle (1pt);
	}
}
\tikzset{
	my arc11/.pic={
		\draw plot [smooth, tension=1.2] coordinates {(-0.2,0) (0,1) (0.2,0)};
	}
}
\tikzset{
	my arc12/.pic={
		\draw plot [smooth, tension=1.2] coordinates {(-0.2,0) (0,-1) (0.2,0)};
	}
}
\tikzset{
	my arc21/.pic={
		\clip (-0.6,0.2) rectangle (0,1.4);
		\draw plot [smooth, tension=1.2] coordinates {(-0.6,0) (0,1.8) (0.6,0)};
	}
}
\tikzset{
	my arc22/.pic={
		\clip (0,0.2) rectangle (0.6,1.4);
		\draw plot [smooth, tension=1.2] coordinates {(-0.6,0) (0,1.8) (0.6,0)};
	}
}
\tikzset{
	my arc23/.pic={
		\clip (-0.6,-0.2) rectangle (0,-1.4);
		\draw plot [smooth, tension=1.2] coordinates {(-0.6,0) (0,-1.8) (0.6,0)};
	}
}
\tikzset{
	my arc24/.pic={
		\clip (0,-0.2) rectangle (0.6,-1.4);
		\draw plot [smooth, tension=1.2] coordinates {(-0.6,0) (0,-1.8) (0.6,0)};
	}
}
\title{\Large{\bf{{Wild attractor exists for symmetric Fibonacci bimodal map}}}}
\author{{Haoyang Ji }\footnotemark[1]\\
{\small{\rm School of Mathematics and Statistics, Zhengzhou University, Zhengzhou 450001, CHINA}}\\
{\small{E-mail: jihymath@zzu.edu.cn}}
}
\date{}
\begin{document}
\maketitle
\vspace{-2cm}
\begin{abstract}
\noindent{ \bf{Abstract.}} We consider smooth symmetric bimodal maps with Fibonacci combinatorics.
We prove that, when the critical order is sufficiently large, the common
$\omega$-limit set of the two critical points is a wild Cantor attractor.
Using a full-family argument, we first prove the existence of smooth symmetric
Fibonacci bimodal maps with any prescribed critical order $\ell>3$.
The proof of the existence of the wild attractor is based on a two-to-one
semi-conjugacy which reduces the generalized renormalization of the bimodal
map to a Fibonacci box mapping with one critical point.
\end{abstract}
\footnoteONLYtext{Date: \date{\today}}
\footnoteONLYtext{\msc{\rm{ 37E05}}}
\footnoteONLYtext{\keywords{\rm{Fibonacci combinatorics, generalized renormalization, semi-conjugacy, wild attractor, full family}}}
\footnoteONLYtext{H. Ji was supported by NSFC Grant No.12301103.}
\section{Introduction}
\label{intro}
Let $f:X\to X$ be a $C^3$ interval map with non-flat critical
points, where $X=[0,1]$. Following \cite{Mil}, a (minimal)
\emph{metric attractor} is a forward invariant compact set $A$
such that its basin
\[
\operatorname{Rel}(A)
=
\{x\in X:\omega(x)\subset A\}
\]
has positive Lebesgue measure, while no proper forward invariant
compact subset of $A$ has this property. By \cite{BL,vSV}, a metric
attractor of $f$ is either a periodic attractor, a cycle of periodic
intervals, or a Cantor set. In the last case, the Cantor set
coincides with the $\omega$-limit set of a recurrent critical point.
There are two types of Cantor attractors: \emph{solenoidal
attractors} and \emph{wild attractors}. A solenoidal attractor
corresponds to an infinitely renormalizable critical point, whereas
a Cantor metric attractor which is not solenoidal is called wild.
In particular, a wild attractor is not a topological attractor.
Milnor \cite{Mil} asked whether wild attractors can exist.

Milnor's attractor problem has been extensively studied, especially
for unimodal maps. For $C^3$ unimodal maps with negative Schwarzian
derivative and a quadratic critical point, it was proved in
\cite{GSS,Lyu} that decay of geometry property holds and hence such maps do
not have wild attractors. This was extended by Shen \cite{S} to
smooth unimodal maps with critical order $1<\ell\leq2$, and by
Li and Shen \cite{LS} to critical order
$\ell\leq2+\epsilon$, where $\epsilon>0$ is small. On the other
hand, for sufficiently large critical order, wild attractors may
exist. The existence was first established in \cite{BKNS} for
unimodal maps with Fibonacci combinatorics.

Fibonacci maps were first introduced by Branner and Hubbard
\cite{BH} in the study of cubic polynomials with one critical point
escaping to infinity, and by Hofbauer and Keller \cite{HK} in the
setting of unimodal maps with slow recurrence as candidates for maps
with wild attractors. The dynamics of Fibonacci interval maps has
 attracted a lot of attentions; see, for example,
\cite{BKNS,KN,Lyu}. It turns out that the metric
properties of Fibonacci unimodal maps depend strongly on the
critical order.

Bruin \cite{B} introduced a broad class of Fibonacci-like unimodal
maps described in terms of the kneading map. For further results on
kneading theory, Fibonacci-like combinatorics, and the associated
symbolic dynamics, see \cite{A,AC,B1,VY}, among others. Lyubich and
Milnor \cite{LM} first observed that renormalization methods can be
applied to Fibonacci unicritical maps, and the corresponding
renormalization theory was subsequently developed in
\cite{Bu,JL,NvS,Sm}.

From the viewpoint of generalized renormalization, Li and Wang
\cite{LW} introduced another class of Fibonacci-like unimodal maps
whose combinatorics need not satisfy Bruin's condition. They proved
that these maps admit no absolutely continuous invariant probability
measure when the critical order is sufficiently large. Note,
however, the absence of an absolutely continuous invariant
probability measure does not imply the existence of a wild
attractor. For extensions to more general combinatorics, see
\cite{JL2}. It was later proved in \cite{Z} that some of the maps
considered in \cite{LW} do have wild attractors for sufficiently
large critical order, using renormalization theory together with the
limit-drift argument introduced in \cite{LSw}.

For multimodal maps, the situation is much more complicated
and less is known. Blokh and Misiurewicz \cite{BM} showed that
the existence of wild attractors is closely related to persistent
recurrence of critical points. The authors of \cite{BKNS} also remarked that their methods could be
used to construct smooth multimodal interval maps with only quadratic
critical points and a wild Cantor attractor, although no detailed
construction was given there. Their method requires
sufficiently many critical points and critical relations,
with the orbit of one critical point passing through other critical
points.

In the bimodal setting, \cite{SV} gives a concrete example of a
non-renormalizable bimodal cubic polynomial with bounded geometry.
This naturally raises the question whether a bimodal cubic
polynomial can have a wild attractor. Vargas \cite{V} introduced
Fibonacci combinatorics for bimodal maps using the natural symmetry
of the two critical orbits and generalized renormalization. This was
further generalized in \cite{JM} to a wide class of Fibonacci-like
bimodal maps. It was also proved in \cite{JM} that cubic
Fibonacci-like polynomials have decay of geometry and hence do not
have wild attractors. The proof uses complex-analytic methods essentially and therefore does not extend to the
smooth case directly.

In this paper, we construct a concrete example of a symmetric
bimodal map with a wild attractor containing both critical points.
To the best of our knowledge, this is the first such example in the
bimodal setting, albeit under a rather restrictive symmetry
assumption. More precisely, we prove that a symmetric bimodal map
with Fibonacci combinatorics has a wild Cantor attractor provided
that the common critical order of its two critical points is
sufficiently large. We also prove the existence of such symmetric
Fibonacci bimodal maps by a modified full family argument.

Let us state our result more precisely. A $C^1$ map $f:[-1,1]\to[-1,1]$ is called \emph{bimodal} if $f(\{-1,1\})=\{-1,1\}$, $f$ has exactly two critical points $-1<d<c<1$, both are turning points, and $f$ is strictly monotone on each of the three intervals determined by $d$ and $c$. If $\{-1, 1\}$ are fixed, $f$ is called {\it positive}; if $\{-1, 1\}$ are interchanged, $f$ is called {\it negative}. A bimodal map $f$ is called {\it symmetric} if
it is odd, that is, $f(-x)=-f(x)$. 

Throughout the paper, we assume that $f$ is $C^3$ outside
$\{d,c\}$ and that both critical points are non-flat. More precisely,
a critical point $\widetilde c\in\{d,c\}$ is called to have
\emph{critical order} $\ell>1$ if there exist $C^3$
diffeomorphisms $\phi$ and $\psi$, defined in neighborhoods of $0$,
such that $\phi(0)=\widetilde c, \psi(0)=f(\widetilde c)$, and $\left|
\psi^{-1}\circ f\circ\phi(x)
\right|
=
|x|^\ell$ for all sufficiently small $|x|$.

\begin{nonametheorem}
There exists $\ell_0>3$ such that the following holds. Let $f:[-1,1]\to[-1,1]$ be a $C^3$ symmetric bimodal map with Fibonacci combinatorics and
two non-flat critical points $d<c$, both of critical order
$\ell\geq\ell_0$. Then
\[
A=\omega(c)=\omega(d)
\]
is a wild Cantor attractor of $f$.
\end{nonametheorem}

The symmetry assumption plays an essential role in our approach. In
particular, it enables us to reduce the dynamics on the twin
principal nest to a box dynamics with a single critical point. The
existence of symmetric Fibonacci bimodal maps of every prescribed
critical order $\ell>3$ is established in Section~3 by a full-family
argument.

Let us briefly describe the main idea of the proof. We make use of
generalized renormalization as in \cite{JL2,JM}. By generalized
renormalization we mean to restrict the first return map on the twin
principal nest (for bimodal maps)  or principal nest (for unimodal maps), to return domains intersection the
critical orbits. Then we can obtain a suitable box mappings. Using the symmetry of $f$, we show that the first generalized
renormalization, which is a box mapping with two critical points, is
semi-conjugate, through a two-to-one projection $\pi$, to a box
mapping $g:I^1\cup J^1\to I^0$ with a unique critical point $c\in I^1$. We then show that $g$ has Fibonacci combinatorics. Thus
the generalized renormalization of the symmetric bimodal map is
reduced, through this semi-conjugacy, to that of a Fibonacci box
map.

For sufficiently large critical order, the known results imply that $A_g=\omega_g(c)$ is a minimal Cantor set whose basin has positive Lebesgue measure.
The semi-conjugacy allows us to pullback this positive-measure basin to the original bimodal maps. We then obtain $A=\omega_f(c)=\omega_f(d)$ as a minimal Cantor set with a positive-measure basin. Since the
Fibonacci combinatorics is non-renormalizable, $A$ is a wild Cantor
attractor of $f$.

\begin{remark}
Using the results of \cite{Z}, our method extends to more general Fibonacci-like combinatorics described by generalized renormalization. One can also define an analogue of Bruin's Fibonacci-like combinatorics for symmetric bimodal maps using kneading maps. Similar results may hold under an extra condition ensuring that both critical points belong to the same minimal Cantor set.
\end{remark}

This article is organized as follows. In Section~2, we review the
combinatorics of Fibonacci unimodal and bimodal maps, together with
the box mappings and generalized renormalizations that will be used. In Section~3, we prove the existence of symmetric Fibonacci
bimodal maps of every prescribed critical order $\ell>3$ by a
full-family argument. In Section~4, we construct the semi-conjugacy
described above and show that the generalized renormalizations of a
symmetric Fibonacci bimodal map give rise to a Fibonacci box map.
Finally, combining with known results on
Fibonacci box mappings, we prove the Main Theorem.

\section{Combinatorics}

In this section, we shall review some useful facts about Fibonacci combinatorics for unimodal and bimodal maps.

Throughout the paper, we use the following conventions. For two points
$a,b\in\mathbb{R}$, we denote by $(a,b)$ the open interval with endpoints
$a$ and $b$, regardless of their order. If $J$ and $J'$ are two intervals
in $\mathbb{R}$, we write $J<J'$ (respectively, $J\leq J'$) if
$y<y'$ (respectively, $y\leq y'$) for every $y\in J$ and $y'\in J'$.
The notation $a<J$, $a\leq J$, $J<a$, and $J\leq a$ is defined analogously
for $a\in\mathbb{R}$.

\subsection{Fibonacci combinatorics and box mappings}

A {\it unimodal} map is a continuous map \( f : [0, 1] \to [0, 1] \), $f(0) = f(1) = 0$, for which there exists a point \( c \in (0, 1) \) such that \( f|_{[0,c]} \) is strictly increasing and \( f|_{[c,1]} \) is strictly decreasing. The point \( c \) is called the \emph{critical point} (sometimes also called the \emph{turning point}), and for all \( i \geq 1 \) we set  
\( c_i = f^i(c). \)

Let \( f^n \) denote the \( n \)-th iterate of \( f \). Let \( J \subset [0, 1] \) be the maximal interval such that \( c \in \partial J \) and \( f^n|_{J} \) is monotone, then \( J \) is called a \emph{central branch} of \( f^n \). An iterate \( n \) is called a \emph{cutting time} if the image of the central branch of \( f^n \) contains \( c \). The cutting times are denoted by \( S_0, S_1, S_2, \dots \), where \( S_0 = 1 \) and \( S_1 = 2 \). Note that the difference between two consecutive cutting times is again a cutting time, see for example \cite{B1}. So we may define the following function
\[
Q: \mathbb{N} \to \mathbb{N} \cup \{0\},
\]
called the \emph{kneading map}, by
\[
S_k - S_{k-1} = S_{Q(k)}.
\]
For the sake of completeness set \( Q(0) = 0 \). If only finitely many cutting times exist, we set \( S_k = \infty \) and \( Q(k) = \infty \) for all sufficiently large \( k \).

\begin{definition}
A unimodal map $f$ is said to have \textbf{Fibonacci combinatorics} (or is called a \textbf{Fibonacci map}) if its kneading map satisfies $Q(k) = \max\{k-2, 0\}$ for all $k \ge 1$; equivalently, $S_k - S_{k-1} = S_{k-2}$ for every $k \ge 2$.
\end{definition}

It follows that the cutting times of a Fibonacci unimodal map \( S_0, S_1, S_2, \dots \) coincides with the Fibonacci sequence \(1, 2, 3, 5, \cdots \). Bruin's class of Fibonacci-like unimodal maps \cite{B} corresponds to the case $Q(k) \geq \max\{k-N, 0\}$, $N \geq 2$.

On the other hand, the combinatorics of a non-renormalizable unimodal map with a recurrent critical point can also be described in terms of the principal nest and generalized renormalizations \cite{Lyu}. In this sense, Fibonacci unimodal map has an equivalent description.

An open interval $T$ is called {\it nice} if $f^n(\partial T) \cap T = \emptyset$ for all $n \geq 0$. Let $D(T)=\{x\in [0,1]: f^k(x)\in T \mbox{ for some }k\geq 1\}$.
The {\it first entry map} $R_T: D(T)\to T$ is defined as $x \to f^{k(x)}(x)$, where $k(x)$ is the {\it entry time} of $x$ into $T$, i.e., the minimal positive integer such that $f^{k(x)}(x)\in T$. The map $R_T|(D(T)\cap T)$ is called the {\it first return map} of $T$. A component of $D(T)$ (resp. $D(T)\cap T$) is called an {\it entry domain} (resp. {\it return domain}) of $T$.

The {\it principal nest}
\begin{equation*}
I^0 \supset I^1 \supset I^2 \supset \ldots
\end{equation*}
of $f$ is defined as follows. First let $I^0 :=(\hat \alpha,\alpha)$, where $\alpha$ is the unique orientation reversing fixed point of $f$ and $f(\hat \alpha) = f(\alpha)$. Then for every $n \geq 1$, define $I^n$ be the return domain to $I^{n-1}$ which contains the critical point. All these intervals $I^n$ are nice. In case that $R_{I^n} : I^{n+1} \to I^{n}$ {\it non-central}, that is, $R_{I^n}(c) \notin I^{n+1}$, let $J^{n+1}$ denote the return domain to $I^n$ containing $R_{I^n}(c)$.

\begin{prop}
A unimodal map $f$ has Fibonacci combinatorics if and only if it satisfies the following properties:
\begin{enumerate}
\item[(1)] $f(c), f^2(c) \notin I^0$ and $f^3(c) \in I^0$; 
\item[(2)] for each $n \geq 1$, $J^n \neq I^n$ and $I^{n-1} \cap \omega(c) \subset I^n \cup J^n$;
\item[(3)] for each $n \geq 1$, $R_{I^{n-1}} |I^n = f^{S_{n+1}}$ and $R_{I^{n-1}} |J^n = f^{S_{n}}$;
\item[(4)] for each $n \geq 1$, $R_{I^{n-1}} (I^n) \supset I^n \ni c$.
\end{enumerate}

\end{prop}

The proof of Proposition 1 follows readily by induction; see, for example, \cite{LW,MS}. One may observe that the essential dynamics of a Fibonacci unimodal map is captured by restricting the first return map to those return domains intersecting the critical orbit. This restriction, which may be viewed as a generalized renormalization, gives rise to a class of box mappings that is more convenient for our purposes.

Let $\mathcal F$ denote the class of $C^1$ mappings $\phi: I_0^1 \cup I_1^1 \to I_0^0$ satisfying the following:
\begin{enumerate}
\item[(1)] $I_0^1, I_1^1$ and $I_0^0$ are open intervals such that $I_0^1 \cap I_1^1 = \emptyset$ and $I_0^1 \cup I_1^1 \subset I_0^0$. Each branch extends continuously to its closure, and $\phi(\partial I_i^1) \subset \partial I_0^0$ for $i=0,1$;
\item[(2)] $\phi : I_0^1 \to I_0^0$ has a unique critical point $c$ while $\phi : I_1^1 \to I_0^0$ is monotone and onto.
\end{enumerate}
Maps from class $\mathcal F$ will be called {\it box mappings}.

\begin{definition}
A box mapping $\phi$ is called {\bf Fibonacci renormalizable} if 
\begin{enumerate}
\item[(1)] $\phi(I_0^1) \supset I_0^1$; 
\item[(2)] $\phi(c) \notin I_0^1$, $\phi^2(c), \phi^3(c) \in I_0^1$.
\end{enumerate}

\end{definition}

The defining orbit configuration is illustrated in
Figure~\ref{fig:fibonacci-box-mapping}.

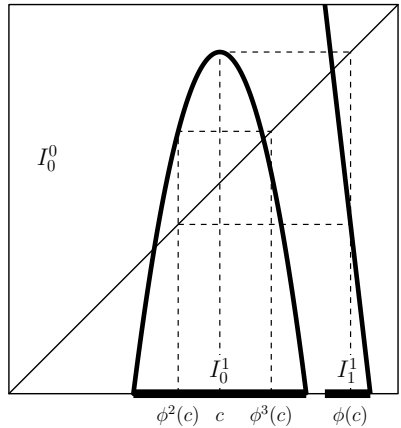
\begin{figure}[H]
\centering
\begingroup
\setlength{\fboxsep}{8pt}
\setlength{\fboxrule}{0pt}
\resizebox{0.35\textwidth}{!}{%
\fbox{%
\begin{tikzpicture}[x=0.82cm,y=0.82cm,line cap=round,line join=round]
  \def\xcrit{5.42}
  \def\xsecond{4.35}
  \def\xfirst{8.78}
  \def\curvature{1.78}
  \pgfmathsetmacro{\xthird}{\xfirst-\curvature*(\xsecond-\xcrit)^2}
  \pgfmathsetmacro{\centralleft}{\xcrit-sqrt(\xfirst/\curvature)}
  \pgfmathsetmacro{\centralright}{\xcrit+sqrt(\xfirst/\curvature)}
  \def\sideleft{8.12}
  \pgfmathsetmacro{\sideright}{(10*\xfirst-\xsecond*\sideleft)/(10-\xsecond)}
  \pgfmathsetmacro{\sidecenter}{(\sideleft+\sideright)/2}

  \draw[line width=0.85pt,line cap=rect] (0,0) rectangle (10,10);
  \draw[line width=0.85pt] (0,0) -- (10,10);

  \draw[line width=2.8pt]
    plot[domain=\centralleft:\centralright,samples=160]
      (\x,{\xfirst-\curvature*(\x-\xcrit)^2});
  \draw[line width=2.8pt,line cap=butt] (\sideleft,10) -- (\sideright,0);

  \begin{scope}[dash pattern=on 3.4pt off 3.4pt,line width=0.72pt,line cap=butt]
    \draw (\xcrit,0) -- (\xcrit,\xfirst) -- (\xfirst,\xfirst);
    \draw (\xfirst,0) -- (\xfirst,\xfirst);
    \draw (\xsecond,0) -- (\xsecond,\xthird) -- (\xthird,\xthird);
    \draw (\xthird,0) -- (\xthird,\xthird);
    \draw (\xsecond,\xsecond) -- (\xfirst,\xsecond);
  \end{scope}

  \draw[line width=5.2pt,line cap=butt] (\centralleft,0) -- (\centralright,0);
  \draw[line width=5.2pt,line cap=butt] (\sideleft,0) -- (\sideright,0);

  \node[anchor=base,font=\large] at (\xsecond,-0.70) {$\phi^{2}(c)$};
  \node[anchor=base,font=\large] at (\xcrit,-0.70) {$c$};
  \node[anchor=base,font=\large] at (\xthird,-0.70) {$\phi^{3}(c)$};
  \node[anchor=base,font=\large] at (\xfirst,-0.70) {$\phi(c)$};
  \node[anchor=base,fill=white,inner sep=1.5pt,font=\Large]
    at (\xcrit,0.42) {$I^{1}_{0}$};
  \node[anchor=base,fill=white,inner sep=1.5pt,font=\Large]
    at (\sidecenter,0.42) {$I^{1}_{1}$};
  \node[anchor=west,font=\Large] at (0.55,6.05) {$I^{0}_{0}$};
\end{tikzpicture}
}%
}
\endgroup
\caption{A Fibonacci-renormalizable box mapping
$\phi:I_0^1\cup I_1^1\to I_0^0$. The critical point $c$ belongs to
$I_0^1$, while $\phi(c)\notin I_0^1$ and
$\phi^2(c),\phi^3(c)\in I_0^1$.}
\label{fig:fibonacci-box-mapping}
\end{figure}

Denote by \( I_0^2 \) the connected component of the domain of the first return map to  
\( I_0^1 \) which contains $c$ and by \( I_1^2 \) the connected component of the domain of the first  
return map to \( I_0^1 \) which contains \( \phi^2(c) \). Note that \( I_0^2 \cap I_1^2 = \emptyset \). The map induced from \(\phi\), denoted by \(\mathcal I\phi\), is defined
as the restriction of the first return map \(R_{I_0^1}\) to
\(I_0^2\cup I_1^2\). Clearly, \(\mathcal I\phi\in\mathcal F\).

If \(\mathcal I^n\phi\) is well defined and Fibonacci renormalizable for
every \(n\geq0\), then we say that \(\phi\) is \emph{infinitely Fibonacci
renormalizable}. There are four configurations: the central branch has either a maximum
or a minimum, and the other branch is either increasing or decreasing.
Reflection of the interval pairs these configurations, leaving two types;
see \cite{Bu,LW,JL}. On the central return domain, $\mathcal I\phi=\phi^2$.
If $\phi|_{I_1^1}$ is increasing, a maximum remains a maximum and a minimum
remains a minimum. If it is decreasing, a maximum becomes a minimum and
a minimum becomes a maximum. For a successive four-step inducing cycle, see Figure~\ref{fig:fibonacci-renormalization-cycle}.

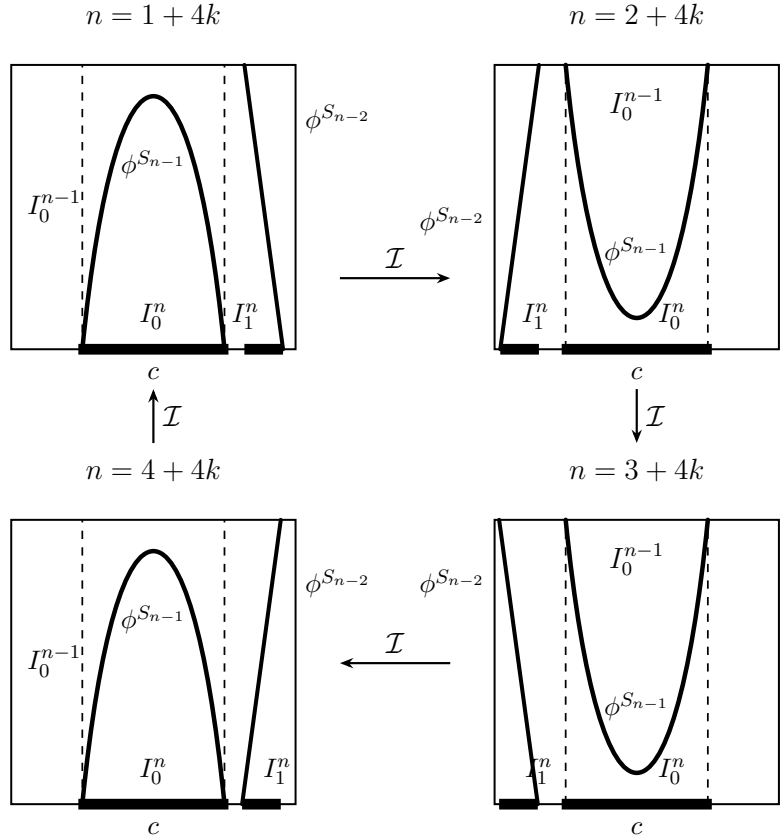
\begin{figure}[H]
\centering
\begingroup
\tikzset{
  panel frame/.style={line width=0.82pt,line cap=rect},
  central branch/.style={line width=1.85pt,line cap=round,line join=round},
  side branch/.style={line width=1.65pt,line cap=round,line join=round},
  domain bar/.style={line width=4.4pt,line cap=butt},
  partition/.style={dash pattern=on 3.4pt off 3.4pt,line width=0.72pt,line cap=butt},
  domain label/.style={font=\normalsize,anchor=base},
  map label/.style={font=\normalsize,inner sep=1pt},
  cycle arrow/.style={-{Stealth[length=5.5pt,width=4.2pt]},line width=0.9pt}
}
\setlength{\fboxsep}{9pt}
\setlength{\fboxrule}{0pt}
\resizebox{0.68\textwidth}{!}{%
\fbox{%
\begin{tikzpicture}[x=0.82cm,y=0.82cm]

  \begin{scope}[shift={(0,8.0)}]
    \draw[panel frame] (0,0) rectangle (5,5);
    \draw[partition] (1.25,0) -- (1.25,5);
    \draw[partition] (3.75,0) -- (3.75,5);
    \draw[central branch] (1.25,0)
      .. controls (1.55,3.00) and (2.02,4.45) .. (2.50,4.45)
      .. controls (2.98,4.45) and (3.45,3.00) .. (3.75,0);
    \draw[side branch] (4.10,5) -- (4.78,0);
    \draw[domain bar] (1.18,0) -- (3.82,0);
    \draw[domain bar] (4.10,0) -- (4.78,0);
    \node[map label] at (2.50,3.25) {$\phi^{S_{n-1}}$};
    \node[map label,anchor=west] at (5.12,4.05) {$\phi^{S_{n-2}}$};
    \node[domain label] at (2.50,0.48) {$I_0^n$};
    \node[domain label] at (4.12,0.48) {$I_1^n$};
    \node[domain label] at (2.50,-0.52) {$c$};
    \node[domain label,anchor=west] at (0.12,2.52) {$I_0^{n-1}$};
    \node[font=\large,anchor=base] at (2.50,5.72) {$n=1+4k$};
  \end{scope}

  \begin{scope}[shift={(8.5,8.0)}]
    \draw[panel frame] (0,0) rectangle (5,5);
    \draw[partition] (1.25,0) -- (1.25,5);
    \draw[partition] (3.75,0) -- (3.75,5);
    \draw[central branch] (1.25,5)
      .. controls (1.55,2.00) and (2.02,0.55) .. (2.50,0.55)
      .. controls (2.98,0.55) and (3.45,2.00) .. (3.75,5);
    \draw[side branch] (0.10,0) -- (0.78,5);
    \draw[domain bar] (1.18,0) -- (3.82,0);
    \draw[domain bar] (0.10,0) -- (0.78,0);
    \node[map label] at (2.50,1.72) {$\phi^{S_{n-1}}$};
    \node[map label,anchor=east] at (-0.12,2.25) {$\phi^{S_{n-2}}$};
    \node[domain label] at (3.12,0.48) {$I_0^n$};
    \node[domain label] at (0.72,0.48) {$I_1^n$};
    \node[domain label] at (2.50,-0.52) {$c$};
    \node[domain label,font=\large] at (2.50,4.15) {$I_0^{n-1}$};
    \node[font=\large,anchor=base] at (2.50,5.72) {$n=2+4k$};
  \end{scope}

  \begin{scope}[shift={(8.5,0)}]
    \draw[panel frame] (0,0) rectangle (5,5);
    \draw[partition] (1.25,0) -- (1.25,5);
    \draw[partition] (3.75,0) -- (3.75,5);
    \draw[central branch] (1.25,5)
      .. controls (1.55,2.00) and (2.02,0.55) .. (2.50,0.55)
      .. controls (2.98,0.55) and (3.45,2.00) .. (3.75,5);
    \draw[side branch] (0.08,5) -- (0.76,0);
    \draw[domain bar] (1.18,0) -- (3.82,0);
    \draw[domain bar] (0.08,0) -- (0.76,0);
    \node[map label] at (2.50,1.72) {$\phi^{S_{n-1}}$};
    \node[map label,anchor=east] at (-0.12,3.88) {$\phi^{S_{n-2}}$};
    \node[domain label] at (3.12,0.48) {$I_0^n$};
    \node[domain label] at (0.78,0.48) {$I_1^n$};
    \node[domain label] at (2.50,-0.52) {$c$};
    \node[domain label,font=\large] at (2.50,4.15) {$I_0^{n-1}$};
    \node[font=\large,anchor=base] at (2.50,5.72) {$n=3+4k$};
  \end{scope}

  \begin{scope}[shift={(0,0)}]
    \draw[panel frame] (0,0) rectangle (5,5);
    \draw[partition] (1.25,0) -- (1.25,5);
    \draw[partition] (3.75,0) -- (3.75,5);
    \draw[central branch] (1.25,0)
      .. controls (1.55,3.00) and (2.02,4.45) .. (2.50,4.45)
      .. controls (2.98,4.45) and (3.45,3.00) .. (3.75,0);
    \draw[side branch] (4.06,0) -- (4.74,5);
    \draw[domain bar] (1.18,0) -- (3.82,0);
    \draw[domain bar] (4.06,0) -- (4.74,0);
    \node[map label] at (2.50,3.25) {$\phi^{S_{n-1}}$};
    \node[map label,anchor=west] at (5.12,3.88) {$\phi^{S_{n-2}}$};
    \node[domain label] at (2.50,0.48) {$I_0^n$};
    \node[domain label] at (4.68,0.48) {$I_1^n$};
    \node[domain label] at (2.50,-0.52) {$c$};
    \node[domain label,anchor=west] at (0.12,2.52) {$I_0^{n-1}$};
    \node[font=\large,anchor=base] at (2.50,5.72) {$n=4+4k$};
  \end{scope}

  \draw[cycle arrow] (5.78,9.25) -- node[above,font=\large] {$\mathcal I$} (7.72,9.25);
  \draw[cycle arrow] (11.00,7.30) -- node[right,font=\large] {$\mathcal I$} (11.00,6.35);
  \draw[cycle arrow] (7.72,2.48) -- node[above,font=\large] {$\mathcal I$} (5.78,2.48);
  \draw[cycle arrow] (2.50,6.35) -- node[right,font=\large] {$\mathcal I$} (2.50,7.30);
\end{tikzpicture}
}%
}
\endgroup
\caption{The four-step orientation cycle of the induced Fibonacci box
mappings $\mathcal I^{n-1}\phi:I_0^n\cup I_1^n\to I_0^{n-1}$. The
central and monotone branches are $\phi^{S_{n-1}}$ and
$\phi^{S_{n-2}}$, respectively. Each arrow represents one inducing step; the last arrow
returns to the first configuration with $k$ replaced by $k+1$.}
\label{fig:fibonacci-renormalization-cycle}
\end{figure}

We shall use the Main Theorem of \cite{BKNS} in the following form.

\begin{thm}
There exists $\ell_0<\infty$ such that the following holds for every $\ell\geq\ell_0$. Let $\phi\in\mathcal F$ be a $C^3$ infinitely Fibonacci renormalizable box mapping with a non-flat critical point $c$ of order $\ell$. Then $A=\omega(c)$ is a minimal Cantor set whose basin ${\rm Rel}(A)$ has positive Lebesgue measure.

\end{thm}

The proof of the Main Theorem in \cite{BKNS} extends to the Fibonacci
box mappings considered here. Indeed, the combinatorial description of
the critical orbit yields the same geometric estimates as those in
\cite[Theorem 3.14]{BKNS}. The induced Markov map can then be constructed
in the same way as in \cite[Section 5]{BKNS} and
\cite[Section 4]{B}. Note that the induced Markov map is constructed
only on $I_0^1$, and its construction depends on the itineraries of
the boundary points $\partial I_0^n$, which are exactly the same in the
present setting. The random walk argument of
\cite[Theorem~5.2]{B} therefore applies and yields the desired result.
Thus, the proof of \cite{BKNS} carries over with only minor
modifications.

\begin{remark}

We emphasize that this argument is only rapid for the Fibonacci
combinatorics of the box mapping. For more general combinatorics, such
as those considered in \cite{Z}, the above argument no longer applies.
In that setting, one needs to use the limit drift method, and only
partial results are currently obtained in \cite{Z}.

\end{remark}

\subsection{Fibonacci bimodal map}

Let $\mathscr B$ denote the class of $C^3$ bimodal maps
$f:[-1,1]\to[-1,1]$ with non-flat critical points. Let $\mathscr B^+$ and $\mathscr B^-$ denote, respectively, the subset of positive and negative bimodal maps from class $\mathscr B$.

Assume that both critical points $d<c$ are recurrent and non-periodic.
Choose a fixed point $p$ with three distinct preimages
$p_1<d<p_2<c<p_3$, and define $U^0=(p_1,p_2)$ and
$V^0=(p_2,p_3)$. For $f\in\mathscr B^+$, such a point exists:
$f(d)>d$ and $f(c)<c$, so there is a unique
fixed point $p\in(d,c)$, with $p_2=p$.

If no fixed point has three distinct preimages, then
$f\in\mathscr B^-$ has a unique fixed point $p\notin[d,c]$,
with $f^{-1}(p)=\{p\}$. Let $J$ be the interval with endpoints
$p$ and $1$ when $p<d$, and with endpoints $-1$ and $p$ when
$p>c$. Then $f^2|_J:J\to J$ is positive bimodal, and we apply
the above construction after an affine change of coordinates.

Define
inductively
\[
U^0\supset U^1\supset U^2\supset\cdots\supset\{d\},
V^0\supset V^1\supset V^2\supset\cdots\supset\{c\},
\]
such that, for each $k\geq1$, $U^k$ and $V^k$ are components of the
domain of the first return map $\phi_k$ to $U^{k-1}\cup V^{k-1}$. The intervals $U^k$ and $V^k$ are called the \emph{critical domains} of $\phi_k$, and the two nested sequences above form the
\emph{twin principal nest}. Let $r_k$ and $t_k$ denote the critical return times defined by $ \phi_k(d)=f^{r_k}(d), \phi_k(c)=f^{t_k}(c)$. The first return map $\phi_k$ is called {\it central} if $ \phi_k(d)\in U^k\cup V^k$ 
or $\phi_k(c)\in U^k\cup V^k$; otherwise it is called {\it non-central}.

\begin{definition}
Let $f$ be a bimodal map with recurrent critical points $d<c$, and
suppose that the critical return times $r_k$ and $t_k$ are defined for
all $k\ge1$. We say that $f$ has \emph{Fibonacci combinatorics} if $r_k=t_k=S_k$, where $S_1=2, S_2=3,
S_{k+1}=S_k+S_{k-1}$.
\end{definition}

According to \cite[Lemma 3.1]{V}, a Fibonacci bimodal map has no central return. Moreover, the description of its kneading sequences in \cite[Section 4]{V} shows that its kneading invariant is combinatorially symmetric in the following sense. In particular, there exists a fixed point $p$ between its critical points.

Let $d<c$ be the two critical points of $f$. The points $d$ and $c$
determine the partition
\[
I_L=(-1, d) , I_M=(d,c), I_R=(c, 1).
\]
For $x\in [-1, 1]$, define its itinerary
\[
\iota_f(x)=\iota_0(x)\iota_1(x)\iota_2(x)\cdots
\]
over the alphabet $\{L,D,M,C,R\}$ by
\[
\iota_n(x)=
\begin{cases}
L,& f^n(x)<d,\\
D,& f^n(x)=d,\\
M,& d<f^n(x)<c,\\
C,& f^n(x)=c,\\
R,& f^n(x)>c.
\end{cases}
\]
The {\it kneading invariant} of $f$ is defined as
\[
K(f)=\bigl(\iota_f(f(d)),\iota_f(f(c))\bigr).
\]
Let $\tau$ be the involution
\[
\tau(L)=R, \tau(R)=L,
\tau(D)=C, \tau(C)=D,
\tau(M)=M,
\]
extended coordinatewise to itineraries.

\begin{definition}
Let $f$ be a bimodal map with critical points $d<c$.
We say that $f$ is \emph{combinatorially symmetric} if $\iota_f(f(c))
=
\tau\bigl(\iota_f(f(d))\bigr)$. 
\end{definition}

The combinatorial symmetry implies that the post-critical points $\phi_k(d)$ and $\phi_k(c)$ always lie in corresponding return domains on opposite sides. For each $k \geq 1$, let $D^k \subset U^{k-1}$ and $C^k \subset V^{k-1}$ be the return domains intersecting $\{\phi_k(d),\phi_k(c)\}$, respectively, and call them the {\it post-critical domains}. Note that the post-critical branches $\phi_k|_{D^k}$ and $\phi_k|_{C^k}$ are monotone and onto. The Fibonacci combinatorics also imposes some constraints on the positions of the post-critical domains and their images. This leads us to consider the following three types of first return maps $\phi_k$:

\begin{itemize}
\item Type $\mathcal A$: if $\phi_k(D^k) = V^{k-1}$, $\phi_k(U^k) \subset V^{k-1}$ and $\phi_k(C^k) = U^{k-1}$, $\phi_k(V^k) \subset U^{k-1}$;
\item Type $\mathcal B$: if $\phi_k(D^k) = V^{k-1}$, $\phi_k(U^k) \subset U^{k-1}$ and $\phi_k(C^k) = U^{k-1}$, $\phi_k(V^k) \subset V^{k-1}$;
\item Type $\mathcal C$: if $\phi_k(D^k) = U^{k-1}$, $\phi_k(U^k) \subset V^{k-1}$ and $\phi_k(C^k) = V^{k-1}$, $\phi_k(V^k) \subset U^{k-1}$.
\end{itemize}

The Fibonacci combinatorics shows that the sequence of first return maps $\phi_1,\phi_2,\phi_3,\ldots$ exhibits a specific sequence of types, as stated in Fact 2.1 below. Together with an analysis of the orientations and the precise positions of the branches, this determines the topological properties of a Fibonacci bimodal map. We therefore subdivide each type $\mathcal A$, $\mathcal B$ and $\mathcal C$ into subtypes $\mathcal A^{ij}$, $\mathcal B^{ij}$ and $\mathcal C^{ij}$, where $i,j\in\{+,-\}$. Here $i=+$ or $i=-$ according as the post-critical branches of $\phi_k$ are orientation-preserving or orientation-reversing, and $j=+$ or $j=-$ according as the critical branch of $\phi_k$ has a local maximum or a local minimum at $d$.

\begingroup
\raggedbottom
\begin{fact}\cite{JM,V}
If $f\in\mathscr B$ has Fibonacci combinatorics, then for every $k\geq1$ the following hold:
\begin{itemize}
\item[1.] The post-critical set $A:=\omega(c)=\omega(d)$ is a minimal Cantor set.
\item[2.] $\phi_k(U^k\cup V^k)\supset U^k\cup V^k$.
\item[3.] Set $\phi_0:=f$, and let $S_0=1$, $S_1=2$ and $S_{n+1}=S_n+S_{n-1}$ for $n\geq1$. Then $\phi_k|_{U^k}$ and $\phi_k|_{V^k}$ are equal to $f^{S_k}$, while $\phi_k|_{D^k}$ and $\phi_k|_{C^k}$ are equal to $f^{S_{k-1}}$. Furthermore,
\[
\phi_k|_{U^k\cup V^k}=\phi_{k-1}^2, \phi_k|_{D^k\cup C^k}=\phi_{k-1}.
\]
\item[4.] The sequence $\phi_1,\phi_2,\phi_3,\ldots$ of first return maps exhibits the sequence of types
\[
\mathcal A^{++}\mathcal B^{-+}\mathcal C^{--}\mathcal A^{-+}\mathcal B^{+-}\mathcal C^{+-}\mathcal A^{+-}\mathcal B^{--}\mathcal C^{-+}\mathcal A^{--}\mathcal B^{++}\mathcal C^{++}\mathcal A^{++}\ldots
\]
or
\[
\mathcal C^{-+}\mathcal A^{--}\mathcal B^{++}\mathcal C^{++}\mathcal A^{++}\mathcal B^{-+}\mathcal C^{--}\mathcal A^{-+}\mathcal B^{+-}\mathcal C^{+-}\mathcal A^{+-}\mathcal B^{--}\mathcal C^{-+}\ldots,
\]
according as $f\in\mathscr B^+$ or $f\in\mathscr B^-$, respectively; see
Figure~\ref{fig:fibonacci-types}.
\end{itemize}
\end{fact}

It is also possible to describe Fibonacci combinatorics in the context of kneading map, see \cite{JM2}.

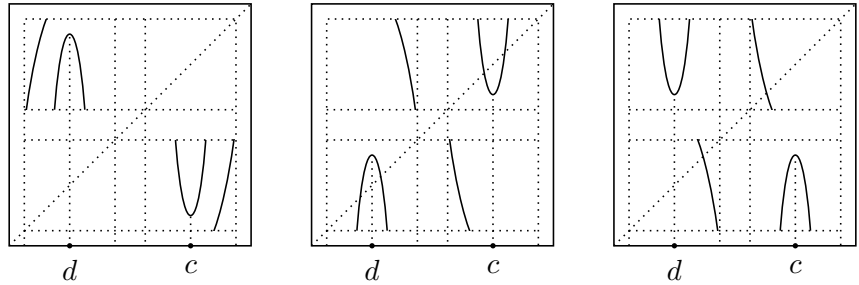
\begin{figure}[H]
		\centering
		
			\begin{tikzpicture}[scale=1, line width=0.6pt, >=stealth]
			\pic at(-6,6) {my squ};
			\pic at($(-6,6)+(-0.8,0)$) {my arc21};
			\pic at($(-6,6)+(-0.8,0.2)$) {my arc11};
			\pic at($(-6,6)+(0.8,-0.2)$) {my arc12};
			\pic at($(-6,6)+(0.8,0)$) {my arc24};
			\draw[dotted] ($(-6,6)+(-0.8,-1.6)$) -- ($(-6,6)+(-0.8,1.2)$);
			\draw[dotted] ($(-6,6)+(0.8,-1.6)$) -- ($(-6,6)+(0.8,-1.2)$);	
		
			\pic at(-2,6) {my squ};
			\pic at($(-2,6)+(-0.8,-1.4)$) {my arc11};
			\pic at($(-2,6)+(-0.8,0)$) {my arc22};
			\pic at($(-2,6)+(0.8,0)$) {my arc23};
			\pic at($(-2,6)+(0.8,1.4)$) {my arc12};
			\draw[dotted] ($(-2,6)+(-0.8,-1.6)$) -- ($(-2,6)+(-0.8,-0.4)$);
			\draw[dotted] ($(-2,6)+(0.8,-1.6)$) -- ($(-2,6)+(0.8,0.4)$);

			\pic at(2,6) {my squ};
			\pic at($(2,6)+(-0.8,1.4)$) {my arc12};
			\pic at($(2,6)+(-0.8,-1.6)$) {my arc22};
			\pic at($(2,6)+(0.8,1.6)$) {my arc23};
			\pic at($(2,6)+(0.8,-1.4)$) {my arc11};
			\draw[dotted] ($(2,6)+(-0.8,-1.6)$) -- ($(2,6)+(-0.8,0.4)$);
			\draw[dotted] ($(2,6)+(0.8,-1.6)$) -- ($(2,6)+(0.8,-0.4)$);
		
			\end{tikzpicture}
		\caption{Types $\mathcal A^{++}$, $\mathcal B^{-+}$, $\mathcal C^{--}$, $\mathcal A^{-+}$, $\mathcal B^{+-}$, $\mathcal C^{+-}$, $\mathcal A^{+-}$, $\mathcal B^{--}$, $\mathcal C^{-+}$, $\mathcal A^{--}$, $\mathcal B^{++}$, $\mathcal C^{++}$}
		\label{fig:fibonacci-types}
\end{figure}

\begin{figure}[H]
		\ContinuedFloat
		\centering
		\begin{tikzpicture}[scale=1, line width=0.6pt, >=stealth]

			\pic at(-6,2) {my squ};
			\pic at($(-6,2)+(-0.8,1.6)$) {my arc23};
		        \pic at($(-6,2)+(-0.8,0.2)$) {my arc11};
		        \pic at($(-6,2)+(0.8,-0.2)$) {my arc12};
			\pic at($(-6,2)+(0.8,-1.6)$) {my arc22};
			\draw[dotted] ($(-6,2)+(-0.8,-1.6)$) -- ($(-6,2)+(-0.8,1.2)$);
			\draw[dotted] ($(-6,2)+(0.8,-1.6)$) -- ($(-6,2)+(0.8,-1.2)$);

		        \pic at(-2,2) {my squ};
			\pic at($(-2,2)+(-0.8,0)$) {my arc21};
			\pic at($(-2,2)+(-0.8,-0.2)$) {my arc12};
			\pic at($(-2,2)+(0.8,0.2)$) {my arc11};
			\pic at($(-2,2)+(0.8,0)$) {my arc24};
			\draw[dotted] ($(-2,2)+(-0.8,-1.6)$) -- ($(-2,2)+(-0.8,-1.2)$);
			\draw[dotted] ($(-2,2)+(0.8,-1.6)$) -- ($(-2,2)+(0.8,1.2)$);

			\pic at(2,2) {my squ};
			\pic at($(2,2)+(-0.8,1.4)$) {my arc12};
			\pic at($(2,2)+(-0.8,0)$) {my arc24};
			\pic at($(2,2)+(0.8,0)$) {my arc21};
			\pic at($(2,2)+(0.8,-1.4)$) {my arc11};
			\draw[dotted] ($(2,2)+(-0.8,-1.6)$) -- ($(2,2)+(-0.8,0.4)$);
			\draw[dotted] ($(2,2)+(0.8,-1.6)$) -- ($(2,2)+(0.8,-0.4)$);

			\pic at(-6,-2) {my squ};
			\pic at($(-6,-2)+(-0.8,1.4)$) {my arc12};
			\pic at($(-6,-2)+(-0.8,1.6)$) {my arc24};
			\pic at($(-6,-2)+(0.8,-1.6)$) {my arc21};
			\pic at($(-6,-2)+(0.8,-1.4)$) {my arc11};
			\draw[dotted] ($(-6,-2)+(-0.8,-1.6)$) -- ($(-6,-2)+(-0.8,0.4)$);
			\draw[dotted] ($(-6,-2)+(0.8,-1.6)$) -- ($(-6,-2)+(0.8,-0.4)$);

			\pic at(-2,-2) {my squ};
			\pic at($(-2,-2)+(-0.8,1.6)$) {my arc23};
			\pic at($(-2,-2)+(-0.8,-0.2)$) {my arc12};
			\pic at($(-2,-2)+(0.8,0.2)$) {my arc11};
			\pic at($(-2,-2)+(0.8,-1.6)$) {my arc22};
			\draw[dotted] ($(-2,-2)+(-0.8,-1.6)$) -- ($(-2,-2)+(-0.8,-1.2)$);
			\draw[dotted] ($(-2,-2)+(0.8,-1.6)$) -- ($(-2,-2)+(0.8,1.2)$);
		
			\pic at(2,-2) {my squ};
			\pic at($(2,-2)+(-0.8,0)$) {my arc23};
			\pic at($(2,-2)+(-0.8,0.2)$) {my arc11};
			\pic at($(2,-2)+(0.8,-0.2)$) {my arc12};
			\pic at($(2,-2)+(0.8,0)$) {my arc22};
			\draw[dotted] ($(2,-2)+(-0.8,-1.6)$) -- ($(2,-2)+(-0.8,1.2)$);
			\draw[dotted] ($(2,-2)+(0.8,-1.6)$) -- ($(2,-2)+(0.8,-1.2)$);

			\pic at(-6,-6) {my squ};
			\pic at($(-6,-6)+(-0.8,1.4)$) {my arc12};
			\pic at($(-6,-6)+(-0.8,0)$) {my arc22};
			\pic at($(-6,-6)+(0.8,0)$) {my arc23};
			\pic at($(-6,-6)+(0.8,-1.4)$) {my arc11};
			\draw[dotted] ($(-6,-6)+(-0.8,-1.6)$) -- ($(-6,-6)+(-0.8,0.4)$);
			\draw[dotted] ($(-6,-6)+(0.8,-1.6)$) -- ($(-6,-6)+(0.8,-0.4)$);
		
			\pic at(-2,-6) {my squ};
			\pic at($(-2,-6)+(-0.8,-1.4)$) {my arc11};
			\pic at($(-2,-6)+(-0.8,1.6)$) {my arc24};
			\pic at($(-2,-6)+(0.8,-1.6)$) {my arc21};
			\pic at($(-2,-6)+(0.8,1.4)$) {my arc12};
			\draw[dotted] ($(-2,-6)+(-0.8,-1.6)$) -- ($(-2,-6)+(-0.8,-0.4)$);
			\draw[dotted] ($(-2,-6)+(0.8,-1.6)$) -- ($(-2,-6)+(0.8,0.4)$);

			\pic at(2,-6) {my squ};
			\pic at($(2,-6)+(-0.8,-1.6)$) {my arc21};
			\pic at($(2,-6)+(-0.8,0.2)$) {my arc11};
			\pic at($(2,-6)+(0.8,-0.2)$) {my arc12};
			\pic at($(2,-6)+(0.8,1.6)$) {my arc24};
			\draw[dotted] ($(2,-6)+(-0.8,-1.6)$) -- ($(2,-6)+(-0.8,1.2)$);
			\draw[dotted] ($(2,-6)+(0.8,-1.6)$) -- ($(2,-6)+(0.8,-1.2)$);

		\end{tikzpicture}
		\caption[]{Continued}
	\end{figure}
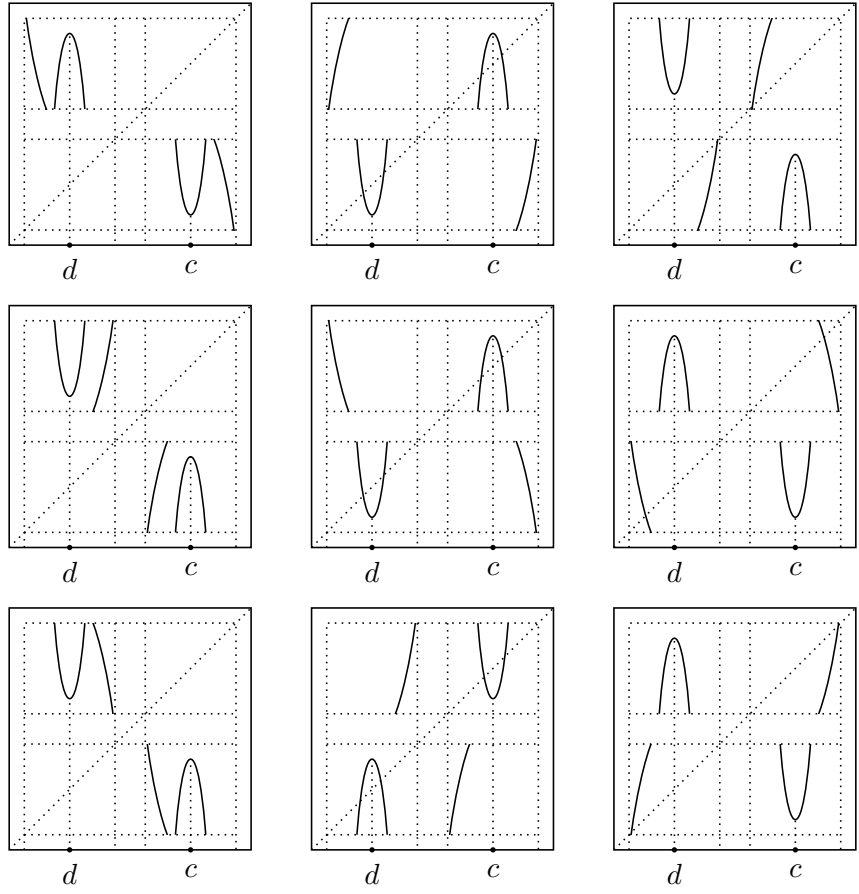
\endgroup

\section{Existence of symmetric Fibonacci bimodal maps}
\label{sec:existence}

In this section we show that, for every prescribed critical order $\ell>3$, there exists a symmetric bimodal map with Fibonacci combinatorics whose two critical points both have critical order $\ell$. The argument is based on the full-family theorem for multimodal maps \cite{GvS,MS}. We first construct an explicit two-parameter full family realizing all admissible bimodal combinatorics. To this end, we parametrize the family by the two critical values. More precisely, for $(u,v)\in\Delta_*$ we construct a bimodal map $f_{u,v}^{(\ell)}$ satisfying
\[
f_{u,v}^{(\ell)}(d)=u, f_{u,v}^{(\ell)}(c)=v.
\]
We then restrict the standard full-family argument to the symmetric subfamily and obtain the existence of symmetric Fibonacci bimodal maps.

For the construction of the two-parameter family, fix $\ell>1$.
We omit the superscript $\ell$ and write $f_{u,v}=f^{(\ell)}_{u,v}$
and $F_a=f_{a,-a}$. Put $d=-1/3$, $c=1/3$, and define
\[
\Delta_*=\{(u,v):-1\le v<u\le1\},\ \Delta=\operatorname{int}\Delta_*.
\]
Define
\[
\rho_\ell(t)=
\begin{cases}
t|t|^{\ell-2},&t\neq0,\\
0,&t=0,
\end{cases}
\ q_\ell(x)=\rho_\ell(x-d)\rho_\ell(x-c).
\]
Then $q_\ell(x)>0$ on $(-1,d)\cup(c,1)$ and $q_\ell(x)<0$ on $(d,c)$. We shall construct a family
\[
f_{u,v}:[-1,1]\to[-1,1], (u,v)\in\Delta_*,
\]
such that
\[
f_{u,v}(-1)=-1, f_{u,v}(d)=u, f_{u,v}(c)=v, f_{u,v}(1)=1.
\]

The factor $q_\ell$ fixes the critical points, their orders, and the signs of the derivative. We shall use three bump functions to prescribe the critical values without changing the local critical behavior.

Choose non-negative $C^\infty$ functions $\eta_L,\eta_M,\eta_R$ satisfying
\begin{enumerate}
\item[(1)] $\operatorname{supp}\eta_L\Subset(-1,d),
\operatorname{supp}\eta_M\Subset(d,c),
\operatorname{supp}\eta_R\Subset(c,1)$;
\item[(2)] $\eta_R(x)=\eta_L(-x), \eta_M(x)=\eta_M(-x)$.
\end{enumerate}
Condition (1) keeps the bump functions away from the critical points, while condition (2) makes their choice compatible with reflection.

After multiplying these functions by positive constants, we may assume that
\[
\int_{-1}^{d}\eta_L(x)q_\ell(x)\,dx = -\int_d^c\eta_M(x)q_\ell(x)\,dx= \int_c^1\eta_R(x)q_\ell(x)\,dx=1.
\]
Put
\[
B_L=\int_{-1}^{d}q_\ell(x)\,dx,
B_M=-\int_d^cq_\ell(x)\,dx,
B_R=\int_c^1q_\ell(x)\,dx.
\]
Then $B_L=B_R>0$ and $B_M>0$.

For $(u,v)\in\Delta_*$, set
\[
\delta_L=u+1, \delta_M=u-v, \delta_R=1-v.
\]
These are the required increases on the left and right branches
and the required decrease on the middle branch, and all are positive on $\Delta_*$. Choose a positive continuous function $\varepsilon:\Delta_*\to(0,\infty)$ which extends continuously to $\overline{\Delta_*}$, satisfies $\varepsilon(-v,-u)=\varepsilon(u,v)$, and
\[
\varepsilon(u,v)B_i<\delta_i, \ i=L,M,R.
\]
For example, one can take
\[
\varepsilon(u,v)=\frac14\frac{\delta_L\delta_M\delta_R}{(1+B_L)(1+B_M)(1+B_R)(1+\delta_L)(1+\delta_M)(1+\delta_R)}.
\]
Define
\begin{align*}
a_L(u,v)&=\delta_L-\varepsilon(u,v)B_L,\\
a_M(u,v)&=\delta_M-\varepsilon(u,v)B_M,\\
a_R(u,v)&=\delta_R-\varepsilon(u,v)B_R,
\end{align*}
so that $a_L(u,v),a_M(u,v),a_R(u,v)>0$. Let
\[
W_{u,v}(x)=\varepsilon(u,v)+\sum_{i\in\{L,M,R\}}a_i(u,v)\eta_i(x),
\]
and define
\begin{equation}
\label{eq:explicit-full-family}
f_{u,v}(x)=-1+\int_{-1}^{x}W_{u,v}(t)q_\ell(t)\,dt.
\end{equation}

We collect the elementary properties of this family in the following lemma.

\begin{lem}
\label{lem:explicit-family}
For every $(u,v)\in\Delta_*$, the map $f_{u,v}$ is a positive bimodal map with critical points $d=-1/3$ and $c=1/3$. Moreover, the critical values are exactly
\[
f_{u,v}(d)=u, f_{u,v}(c)=v,
\]
and both critical points have critical order $\ell$. The map
\[
\Delta_*\ni(u,v)\longmapsto f_{u,v}\in C^1([-1,1])
\]
extends continuously to $\overline{\Delta_*}$. Furthermore,
\begin{equation}
\label{eq:family-symmetry}
f_{-v,-u}(-x)=-f_{u,v}(x).
\end{equation}
\end{lem}

\begin{proof}
Differentiating \eqref{eq:explicit-full-family} gives
\[
Df_{u,v}(x)=W_{u,v}(x)q_\ell(x).
\]
Since $W_{u,v}>0$, we have $Df_{u,v}>0$ on $(-1,d)\cup(c,1)$ and $Df_{u,v}<0$ on $(d,c)$. Thus $d$ is a local maximum, $c$ is a local minimum, and these are the only critical points.

By the definition of $f_{u,v}$ and the normalization of $\eta_L,\eta_M,\eta_R$, we have
\begin{align*}
f_{u,v}(-1)&=-1,\\
f_{u,v}(d)&=-1+\varepsilon(u,v)B_L+a_L(u,v)=u,\\
f_{u,v}(c)&=u-\varepsilon(u,v)B_M-a_M(u,v)=v,\\
f_{u,v}(1)&=v+\varepsilon(u,v)B_R+a_R(u,v)=1.
\end{align*}
Since each branch is monotone, $f_{u,v}([-1,1])=[-1,1]$.

Since $W_{u,v}\equiv\varepsilon(u,v)>0$ near $d$ and $c$, integration of the local expression for $q_\ell$ gives, for $|t|$ small,
\[
u-f_{u,v}(d+t)=|t|^\ell H_d(t),\ f_{u,v}(c+t)-v=|t|^\ell H_c(t),
\]
where $H_d$ and $H_c$ are smooth and strictly positive near $0$ (with $(u,v)$ fixed). Thus both critical points have critical order $\ell$.

Since $\ell>1$, the function $q_\ell$ is continuous. The coefficients of $W_{u,v}$ extend continuously to $\overline{\Delta_*}$, so both $f_{u,v}$ and $Df_{u,v}=W_{u,v}q_\ell$ depend continuously on $(u,v)$ in the uniform norm. This proves the $C^1$ extension.

Finally, $q_\ell$ is even. Moreover,
\[
\delta_L(-v,-u)=\delta_R(u,v), \delta_M(-v,-u)=\delta_M(u,v), \delta_R(-v,-u)=\delta_L(u,v).
\]
Using $B_L=B_R$, the symmetry of $\varepsilon$, and the choice $\eta_R(x)=\eta_L(-x)$ and $\eta_M(x)=\eta_M(-x)$, we obtain
\[
W_{-v,-u}(-x)=W_{u,v}(x).
\]
Together with the endpoint normalization, this gives
\[
f_{-v,-u}(-x)=-f_{u,v}(x).
\]
\end{proof}

The family in Lemma~\ref{lem:explicit-family} is full. In the notation of \cite{GvS}, the space of admissible critical-value pairs is $V=\Delta_*$, including the edges $u=1$ and $v=-1$. The critical-value map is
\[
\operatorname{CV}:\Delta_*\longrightarrow V, \ \operatorname{CV}(u,v)=\left(f_{u,v}(d),f_{u,v}(c)\right)=(u,v).
\]
Thus $\operatorname{CV}$ is the identity. Since the turning points are fixed and the family extends continuously to $\overline{\Delta_*}$ in $C^1$, \cite[Theorem~1]{GvS} shows that $\{f_{u,v}\}_{(u,v)\in\Delta_*}$ is full.

\begin{prop}
\label{prop:existence-symmetric-fibonacci}
For each $\ell>3$, there exists $a_\ell\in(0,1)$ such that
$F_{a_\ell}:=f_{a_\ell,-a_\ell}$ is a positive symmetric
bimodal map with Fibonacci combinatorics. Its critical points are
$d=-1/3$ and $c=1/3$, both of critical order $\ell$.
\end{prop}

\begin{proof}
Fix $\ell>3$. Since the family $\{f_{u,v}\}_{(u,v)\in\Delta_*}$ in Lemma~\ref{lem:explicit-family} is full, there is
$(u_0,v_0)\in\Delta_*$ such that $g:=f_{u_0,v_0}$ has the
Fibonacci kneading invariant. If $u_0=1$ or $v_0=-1$, a critical point maps to a fixed endpoint, contrary to this kneading invariant. Hence $(u_0,v_0)\in\Delta$. Note that $g$ is not assumed to be symmetric. Since $\ell>3$, $g$ is $C^3$ with non-flat critical points. Hence, by \cite[Theorem~A]{vSV}, $g$ has no wandering
intervals.
Since $g$ has Fibonacci combinatorics and no wandering intervals, both critical points are recurrent.

Let $U^0\supset U^1\supset\cdots$ and
$V^0\supset V^1\supset\cdots$ be the twin principal nest of $g$, where
$d\in U^k$ and $c\in V^k$, and let $\phi_k$ be the first return map to
$U^{k-1}\cup V^{k-1}$. By Fact 2.1, 
\begin{equation}
\phi_k|_{U^k}=g^{S_k}, \ \phi_k|_{V^k}=g^{S_k},
\label{eq:return-rule}
\end{equation}
where $S_1=2$, $S_2=3$, and $S_{k+1}=S_k+S_{k-1}$. Since $\omega(c) = \omega(d)$ is a minimal Cantor set and non-existence of wandering intervals, we have
\begin{equation}
|U^k| \longrightarrow0, \ |V^k| \longrightarrow0.
\label{eq:lengths-vanish}
\end{equation}

\noindent{\bf Step 1. Approximation of $g$.}

We shall approximate $g$ by maps $g_k$ whose critical orbits are finite and follows the initial combinatorics of $g$ for increasingly many iterates. For each $k$, let
\begin{equation}
\begin{aligned}
\alpha_k&=\begin{cases}
c,&\text{if $\phi_k$ is of Type $\mathcal A$ or Type $\mathcal C$},\\
d,&\text{if $\phi_k$ is of Type $\mathcal B$},
\end{cases}\\[4pt]
\beta_k&=\begin{cases}
d,&\text{if $\phi_k$ is of Type $\mathcal A$ or Type $\mathcal C$},\\
c,&\text{if $\phi_k$ is of Type $\mathcal B$}.
\end{cases}
\end{aligned}
\label{eq:alpha-beta}
\end{equation}
It is clear that in Types $\mathcal A$ and $\mathcal C$, 
$g^{S_k}(d)\in V^{k-1}$ and $g^{S_k}(c)\in U^{k-1}$, while in Type $\mathcal B$
the two target domains are reversed.
Hence, by \eqref{eq:lengths-vanish},
\begin{equation}
|g^{S_k}(d)-\alpha_k|+|g^{S_k}(c)-\beta_k|\longrightarrow0.
\label{eq:cv-approach}
\end{equation}
Set 
\[
p_k=g^{S_k-1}(d) \mbox{ and } q_k=g^{S_k-1}(c).
\]
The following claim shows how to close the critical orbits while
preserving the symmetry of their left-to-right order.

\begin{claim}
\label{claim:symmetric-closing}
For every sufficiently large $k$, there is a bimodal map $g_k$ with
the same turning points and the same orientation on each lap as $g$
such that $\|g_k-g\|_\infty\longrightarrow 0$ as $k\to\infty$ and the following properties hold:
\begin{enumerate}[label=(\arabic*),ref=\arabic*,itemsep=6pt,parsep=0pt,topsep=6pt]
\item\label{item:closing-values} $g_k(p_k)=\alpha_k$, $g_k(q_k)=\beta_k$;
\item\label{item:closing-orbits} $g_k^j(d)=g^j(d)$, $g_k^j(c)=g^j(c)$ for $0\le j<S_k$;
\item\label{item:closing-terminal} $g_k^{S_k}(d)=\alpha_k$, $g_k^{S_k}(c)=\beta_k$.
\end{enumerate}
Moreover, the pairing $g_k^j(d)\leftrightarrow g_k^j(c)$, $0\le j<S_k$, has the following properties:
\begin{enumerate}[label=(\arabic*),ref=\arabic*,start=4,itemsep=6pt,parsep=0pt,topsep=6pt]
\item\label{item:closing-order} It reverses the order of all marked points; in particular, for $0\le i,j<S_k$,
\[
g_k^i(d)<g_k^j(c) \Longrightarrow g_k^i(c)>g_k^j(d).
\]
\item\label{item:closing-pairing} For $0\le j<S_k-1$, $g_k$ sends the pair $(g_k^j(d),g_k^j(c))$ to $(g_k^{j+1}(d),g_k^{j+1}(c))$, while $(p_k,q_k)$ is sent to $(c,d)$ in Types $\mathcal A$ and $\mathcal C$, and to $(d,c)$ in Type $\mathcal B$.
\end{enumerate}
\end{claim}

\begin{proof}

We first prove (\ref{item:closing-values})--(\ref{item:closing-terminal}) by modifying $g$ near $p_k$ and $q_k$ while keeping its values at all earlier critical-orbit points unchanged.

By the first-return property, no critical-orbit point $g^j(d)$ or $g^j(c)$ with $0\le j<S_k$ lies strictly between $g(p_k)$ and $\alpha_k$,
or between $g(q_k)$ and $\beta_k$. Let $U_{k,d}$ and $U_{k,c}$ denote the monotonicity intervals of $g$
containing $p_k$ and $q_k$, respectively.
For sufficiently large $k$, choose small disjoint intervals
\[
K_{k,d}\subset g(U_{k,d}), \ K_{k,c}\subset g(U_{k,c}),
\]
such that
\[
\{g(p_k),\alpha_k\}\subset\operatorname{int}K_{k,d},
 \ 
\{g(q_k),\beta_k\}\subset\operatorname{int}K_{k,c}.
\]
Choose these intervals so that they contain none of the earlier orbit values $g^j(d),g^j(c)$, $1\le j<S_k$. The absence of these values between the respective endpoints, together with \eqref{eq:cv-approach}, allows us to choose $|K_{k,d}|,|K_{k,c}|\to0$.

Choose an increasing $C^3$ diffeomorphism
$h_{k,d}:K_{k,d}\to K_{k,d}$ that equals the identity near the boundary and satisfies $h_{k,d}(g(p_k))=\alpha_k$.
On the branch containing $p_k$, replace $g(x)$ by $h_{k,d}(g(x))$ wherever $g(x)\in K_{k,d}$, leaving $g$ unchanged elsewhere.
Make the analogous modification near $q_k$, sending $g(q_k)$ to $\beta_k$. These modifications preserve the turning points, the lap orientations, and the earlier orbit, giving properties (\ref{item:closing-values})--(\ref{item:closing-terminal}) and $\|g_k-g\|_\infty\leq\max\{|K_{k,d}|,|K_{k,c}|\}\to0$. Both critical points are then periodic and belong to attracting cycles.

To prove (\ref{item:closing-order}), we use the combinatorial symmetry of $g$: $\tau$ (given in subsection 2.2) exchanges the itineraries of corresponding critical-orbit points. Until the itineraries of $x$ and $y$ first differ, the corresponding
branches have the same direction of monotonicity. At the first difference,
$\tau$ puts the left-hand symbol on the right and the right-hand symbol
on the left. Thus the pairing reverses the order of points with distinct itineraries.

The itinerary map $x\mapsto\iota_g(x)$ is one-to-one on the critical orbits of $g$. Indeed, if distinct points $x,y$ on critical orbits had the same itinerary, every iterate of $g$ would be monotone on the closed interval with endpoints $x,y$. By the Homterval Lemma \cite[Lemma~II.3.1]{MS}, this interval would either be wandering or have every point in the basin of a periodic orbit. The first alternative is excluded by the absence of wandering intervals. The second is also impossible: $x$ and $y$ are recurrent but not periodic, whereas a recurrent point that tends to a periodic orbit must itself belong to that orbit.
Thus the pairing of the two critical orbits is well-defined and reverses their actual order.
By (\ref{item:closing-orbits}), the marked points of $g_k$ coincide with those of $g$, so this pairing also reverses their order, proving (\ref{item:closing-order}).

Finally, for $0\le j<S_k-1$, applying $g_k$ to $(g_k^j(d),g_k^j(c))$ gives $(g_k^{j+1}(d),g_k^{j+1}(c))$.
At the terminal pair $(p_k,q_k)$, property (\ref{item:closing-values}) gives the image $(\alpha_k,\beta_k)$, which is $(c,d)$ in Types $\mathcal A$ and $\mathcal C$ and $(d,c)$ in Type $\mathcal B$ by \eqref{eq:alpha-beta}.
This proves (\ref{item:closing-pairing}) and completes the proof.
\end{proof}

\noindent{\bf Step 2. Critical finite case.}

Write 
\[
P_k=\bigcup_{j\ge0}\{g_k^j(d),g_k^j(c)\}
=\{z_1<\cdots<z_{N_k}\}.
\]
Define $\pi_k:\{1,\ldots,N_k\}\to\{1,\ldots,N_k\}$ by $g_k(z_i)=z_{\pi_k(i)}$, and set $\sigma_k(i)=N_k+1-i$.
By Claim~\ref{claim:symmetric-closing}(\ref{item:closing-order}), $z_{\sigma_k(i)}$ is the point paired with $z_i$; property (\ref{item:closing-pairing}) then gives
\begin{equation}
\sigma_k\circ\pi_k=\pi_k\circ\sigma_k.
\label{eq:pi-sigma}
\end{equation}

We now apply Step~1 of the proof of the full-family theorem
\cite[Chapter~II, Section~4]{MS} to this finite critical-orbit
combinatorics. As in that proof, we collapse the components of the basins of inessential periodic attractors of $g_k$, together with non-degenerate intervals of periodic points and their preimages. Every point of $P_k$ lies on a superattracting cycle containing $d$ or $c$, so it belongs to none of these sets. Thus the reduction does not identify distinct points of $P_k$ and preserves their order, the map $\pi_k$, and the pairing $\sigma_k$.

Moreover, $\pi_k$ is a permutation of $\{1,\ldots,N_k\}$. Let $t_d,t_c$ be the indices for which $z_{t_d}=d$ and $z_{t_c}=c$.
Every cycle of $\pi_k$ contains $t_d$ or $t_c$; equivalently, for each $i$ there is an integer $n\ge0$ such that $\pi_k^n(i)\in\{t_d,t_c\}$, or $g_k^n(z_i)\in\{d,c\}$.
Hence, for each adjacent pair $z_i<z_{i+1}$, we can choose $n\ge0$ such that
\[
g_k^n(z_i)\in\{d,c\}, \ g_k^n(z_{i+1})\ne g_k^n(z_i),
\]
where the inequality follows from the injectivity of $\pi_k^n$.
Thus the iterated pair has a turning point as one endpoint and a distinct marked point as the other. This is precisely
the claim used in Step~1 of \cite[Theorem~II.4.1]{MS} to prove its
boundary estimate.

Let
\[
W_k=\{(x_1,\ldots,x_{N_k}):-1<x_1<\cdots<x_{N_k}<1\}
\]
be the corresponding configuration simplex, and define the involution
\begin{equation}
\begin{aligned}
\mathcal R_k:W_k&\longrightarrow W_k, \ (\mathcal R_kx)_i=-x_{\sigma_k(i)},\\
\mathcal R_k(x_1,\ldots,x_{N_k})&=(-x_{N_k},\ldots,-x_1), \ \mathcal R_k^2=\mathrm{id}_{W_k}.
\end{aligned}
\label{eq:reflect}
\end{equation}
Every point of $P_k$ is paired with a distinct point of $P_k$ under the correspondence in Claim~\ref{claim:symmetric-closing}, so $\sigma_k$ has no fixed point and $N_k$ is even.
Set $m_k=N_k/2$, the number of pairs. The fixed-point set
$W_k^{\rm sym}:=\operatorname{Fix}(\mathcal R_k)$ consists exactly of
the points
\[
(-y_{m_k},\ldots,-y_1,y_1,\ldots,y_{m_k})
\ (0<y_1<\cdots<y_{m_k}<1),
\]
so it is naturally an open $m_k$-simplex.

We define the pullback map $T_k:W_k\to W_k$ as in \cite{MS}. For $x=(x_1,\ldots,x_{N_k})\in W_k$, set $u=x_{\pi_k(t_d)}$ and $v=x_{\pi_k(t_c)}$.
Since $g_k(d)>g_k(c)$, we have $\pi_k(t_d)>\pi_k(t_c)$, so $(u,v)\in\Delta$.
The critical-value map in Lemma~\ref{lem:explicit-family} is the identity, so $f_{u,v}$ is the unique member of the family with these critical values.
Define $T_k(x)=(y_1,\ldots,y_{N_k})$ by $y_{t_d}=d$, $y_{t_c}=c$ and, for $i\notin\{t_d,t_c\}$, by
\[
f_{u,v}(y_i)=x_{\pi_k(i)}, \ 
y_i\in
\begin{cases}
(-1,d),&i<t_d,\\
(d,c),&t_d<i<t_c,\\
(c,1),&i>t_c.
\end{cases}
\]
The prescribed order of the marked points ensures that
$x_{\pi_k(i)}$ lies in the image of this branch.
Since the branch is strictly monotone, there is a unique point $y_i$
on it satisfying $f_{u,v}(y_i)=x_{\pi_k(i)}$. The branch orientations also give $y_1<\cdots<y_{N_k}$, so $T_k(x)\in W_k$; the inverse branches depend continuously on $x$, and hence $T_k$ is continuous.

If $x\in W_k^{\rm sym}$, then \eqref{eq:pi-sigma} and
$\sigma_k(t_d)=t_c$ give
$x_{\pi_k(t_c)}=-x_{\pi_k(t_d)}$. Thus the two prescribed critical
values are $(a,-a)$ with $0<a<1$, so the pullback uses exactly the map $f_{a,-a}$.

For any $x\in W_k$, write $y=T_k(x)$ and $\widehat x=\mathcal R_kx$.
Set $u=x_{\pi_k(t_d)}$ and $v=x_{\pi_k(t_c)}$.
Then the parameters corresponding to $\widehat x$ are $(-v,-u)$.
By \eqref{eq:family-symmetry} and \eqref{eq:pi-sigma},
\[
f_{-v,-u}(-y_i)=-f_{u,v}(y_i)=-x_{\pi_k(i)}
=\widehat x_{\pi_k(\sigma_k(i))}.
\]
Since reflection exchanges the outer branches and fixes the middle branch, $-y_i$ lies in the required branch for $\sigma_k(i)$ when $i\notin\{t_d,t_c\}$.
Uniqueness of the preimage gives $T_k(\widehat x)_{\sigma_k(i)}=-y_i$; the same equality holds at $t_d,t_c$ since $c=-d$.
Hence $T_k\circ\mathcal R_k=\mathcal R_k\circ T_k$, and therefore
\begin{equation}
T_k(W_k^{\rm sym})\subset W_k^{\rm sym}.
\label{eq:t-invariant}
\end{equation}

Put $x_0=-1$, $x_{N_k+1}=1$, and 
\[
\delta_k(x)=\min_{0\le i\le N_k}(x_{i+1}-x_i).
\]
For $x\in W_k^{\rm sym}$, symmetry gives $x_{i+1}-x_i=x_{N_k-i+1}-x_{N_k-i}$, so each gap equals its reflected gap.
Thus approaching the relative boundary of $W_k^{\rm sym}$ is equivalent to $\delta_k(x)\to0$, as for the boundary of $W_k$.
In the Euclidean metric, uniformly in $k$ and $x\in W_k^{\rm sym}$,
\begin{equation}
\frac12\,\operatorname{dist}(x,\partial W_k)
\le
\operatorname{dist}\bigl(x,\partial_{\rm rel}W_k^{\rm sym}\bigr)
\le
2\,\operatorname{dist}(x,\partial W_k).
\label{eq:dist-bound}
\end{equation}
Indeed, each distance lies between $\delta_k(x)/\sqrt2$ and $\sqrt2\delta_k(x)$.
For the relative boundary, a smallest gap is closed together with its reflected gap; the central gap is closed on its own.

By \eqref{eq:dist-bound} and \cite[Lemma~II.4.1]{MS}, for any sequence $x_n\in W_k^{\rm sym}$ approaching the relative boundary,
\[
\frac{\lvert T_k(x_n)-x_n\rvert}
{\operatorname{dist}(x_n,\partial_{\rm rel}W_k^{\rm sym})}
\longrightarrow\infty.
\]
The direction of $T_k(x)-x$ points inward as $x\in W_k^{\rm sym}$ approaches the relative boundary, in the sense used in the proof of \cite[Lemma~II.4.2]{MS}.
By \eqref{eq:t-invariant}, the deformation argument applies in $W_k^{\rm sym}$ and gives a fixed point $x_k^*\in W_k^{\rm sym}$.
The corresponding parameter is $(a_k,-a_k)$ with $0<a_k<1$, so $F_k:=F_{a_k}$ realizes the finite critical-orbit combinatorics of $g_k$.
Since the unchanged initial segments of the critical orbits have no critical relations, both critical points have period $S_k$ in Type $\mathcal B$ and $2S_k$ in Types $\mathcal A$ and $\mathcal C$.
These periods tend to infinity as $k \to \infty$.

\noindent{\bf Step 3. Passage to the limit.}

Passing to a subsequence, assume $a_k\to a_\ell\in[0,1]$. By Lemma~\ref{lem:explicit-family}, $F_k\to F_{a_\ell}$ in $C^1$.
We shall show that $F_{a_\ell}$ has Fibonacci combinatorics. As in Step~3 of the proof of \cite[Theorem~II.4.1]{MS}, we first show that distinct points on the critical orbits of $g$ remain distinct in the limit.

\begin{claim}
For integers $r,s\ge0$ and $\xi,\zeta\in\{d,c\}$, if $g^r(\xi)\ne g^s(\zeta)$, then $F_{a_\ell}^r(\xi)\ne F_{a_\ell}^s(\zeta)$.
Moreover, $F_{a_\ell}^j(\xi)\in(-1,1)$ for every $j\ge0$ and $\xi\in\{d,c\}$.
\end{claim}

\begin{proof}
Fix $r,s$ and $\xi,\zeta$ as in the claim. For $S_k>\max\{r,s\}$, Claim~\ref{claim:symmetric-closing}(\ref{item:closing-orbits}) and the realization of the marked points give $g^r(\xi)<g^s(\zeta)\Rightarrow F_k^r(\xi)<F_k^s(\zeta)$.
Passing to the limit gives
$F_{a_\ell}^r(\xi)\le F_{a_\ell}^s(\zeta)$.
Suppose that equality holds, we shall derive a contradiction.

Let $J$ be the open interval between $g^r(\xi)$ and $g^s(\zeta)$, which is not a homterval. Let $n\ge0$ be the first time that $g^n(J)$ contains $d$ or $c$ in its interior. All earlier images lie in monotonicity intervals.

Let $J_k$ be the interval between $F_k^r(\xi)$ and $F_k^s(\zeta)$.
For $S_k>\max\{r,s\}+n$, Claim~\ref{claim:symmetric-closing}(\ref{item:closing-orbits}) and the order of the marked points imply that $F_k^n(J_k)$ contains the same turning point.
By assumption, $|J_k|\to0$. Since $F_k\to F_{a_\ell}$ in $C^1$, the maps $F_k^n$ are equicontinuous, so $|F_k^n(J_k)|\to0$.
Thus both endpoints converge to that turning point.

Denote this turning point by $\chi\in\{d,c\}$. Then $F_{a_\ell}^{r+n}(\xi)=F_{a_\ell}^{s+n}(\zeta)=\chi$. At least one of $r+n$ and $s+n$ is positive: otherwise $r=s=n=0$ and $J=(d,c)$, which contains no turning point. Thus some positive iterate of a critical point is a critical point. Since every $F_k$ is odd, so is $F_{a_\ell}$; as $d=-c$, this gives $F_{a_\ell}^{m}(c)=\pm c$ for some $m\ge1$, and hence $F_{a_\ell}^{2m}(c)=c$.

Let $p$ be the period of $c$ under $F_{a_\ell}$. Its cycle is superattracting.
Choose a small open interval $H$ around a point of the cycle such that
\[
F_{a_\ell}^p(\overline H)\subset H, \  \sup_{x\in\overline H}|D(F_{a_\ell}^p)(x)|<\frac12.
\]
Thus $H$ is a trapping neighborhood for the return map $F_{a_\ell}^p$: every $p$ iterates stay in $H$ and converge to its fixed point.
By $C^1$ convergence, for all sufficiently large $k$ we also have $F_k^p(\overline H)\subset H$ and $\sup_{\overline H}|D(F_k^p)|<3/4$.
Choose $N\ge0$ with $F_{a_\ell}^N(c)\in H$. Since $F_k^N(c)\to F_{a_\ell}^N(c)$, the same $N$ gives $F_k^N(c)\in H$ for all sufficiently large $k$.
The contraction $F_k^p|_{\overline H}$ therefore attracts this point to its unique fixed point, which belongs to a cycle of $F_k$ of period dividing $p$.
But it is itself periodic, with period at least $S_k\to\infty$; a periodic point attracted to a periodic cycle must belong to that cycle.
This contradiction shows that the two limiting points are distinct.

Suppose $F_k^r(\xi)\to e$ for some fixed $r\ge0$, $\xi\in\{d,c\}$, and $e\in\{-1,1\}$.
Let $J$ be the interval between $g^r(\xi)$ and $e$.
The endpoint is fixed by Lemma~\ref{lem:explicit-family}.
The interval is not a homterval: it cannot wander with a fixed endpoint, and its recurrent, non-periodic endpoint cannot tend to a periodic orbit.
Let $n$ be the first time that $g^n(J)$ contains a turning point $\chi$ in its interior. For all sufficiently large $k$, the corresponding interval between $F_k^r(\xi)$ and $e$ has an $n$-th image containing $\chi$. Its diameter tends to zero, while its endpoint $F_k^n(e)=e$ is fixed. This would give $\chi=e$, contrary to $\chi\in\{d,c\}\subset(-1,1)$.
This proves the claim.
\end{proof}

By the claim, distinct critical-orbit points of $g$ remain distinct and keep their order in the limit.
Thus $F_{a_\ell}$ has the same kneading invariant as $g$.
If $a_\ell=0$, the critical values coincide; if $a_\ell=1$, they are endpoints. Both cases contradict the claim, so $0<a_\ell<1$. Then $F_{a_\ell}$ has Fibonacci combinatorics.

Finally, \eqref{eq:family-symmetry} gives $F_{a_\ell}(-x)=-F_{a_\ell}(x)$.
By Lemma~\ref{lem:explicit-family}, $F_{a_\ell}$ is positive and has exactly two critical points, $d=-1/3$ and $c=1/3$, both of order $\ell$.
\end{proof}

For each even integer $\ell \geq 4$, one can construct a polynomial family in a similar way. Set $m=\ell-1$ and define
\[
C_m=\int_0^1[t(1-t)]^m\,dt, \ A_m(s)=\int_0^s[t(t+1)]^m\,dt \ (s\ge0).
\]
For $(u,v)\in\Delta_*$, let $\alpha,\beta>0$ be determined by
\[
A_m(\alpha)=C_m\frac{u+1}{u-v}, \ A_m(\beta)=C_m\frac{1-v}{u-v}.
\]
These numbers are unique since $A_m$ is strictly increasing from $0$ to $\infty$. Set
\[
h=\frac{2}{1+\alpha+\beta},\ d_{u,v}=-1+\alpha h,\ c_{u,v}=1-\beta h,
\]
so that $c_{u,v}-d_{u,v}=h$, and put $\lambda=(u-v)/(C_mh^{2m+1})$. Define
\[
P_{u,v}(x)=-1+\lambda\int_{-1}^x[(t-d_{u,v})(t-c_{u,v})]^m\,dt.
\]
Direct computation gives
\[
P_{u,v}(-1)=-1, \ P_{u,v}(d_{u,v})=u, \ P_{u,v}(c_{u,v})=v, \ P_{u,v}(1)=1.
\]
Since $m$ is odd and $P_{u,v}'(x)=\lambda(x-d_{u,v})^m(x-c_{u,v})^m$, this is a positive bimodal polynomial of degree $2\ell-1$, with both critical points of order $\ell$.

The critical-value map is the identity, and the family extends continuously to $\overline\Delta_*$ in $C^1$. Thus it is a full family. Moreover, reflection exchanges $\alpha$ and $\beta$, giving
\[
P_{-v,-u}(-x)=-P_{u,v}(x).
\]
For $u=-v=a$, we have $\alpha=\beta$, so $P_{a,-a}$ is odd, with critical points $-b_a,b_a$, where $b_a=(1+2\alpha)^{-1}$. The finite approximation and limiting argument in Proposition~\ref{prop:existence-symmetric-fibonacci}, with the critical points now depending on the parameter, gives a symmetric Fibonacci map in this family.

These polynomials have negative Schwarzian derivative away from their critical points. For $0<a<1$, the derivative $P_{a,-a}'$ is strictly increasing on $(b_a,1)$, so
\[
P_{a,-a}'(-1)=P_{a,-a}'(1)>
\frac{P_{a,-a}(1)-P_{a,-a}(b_a)}{1-b_a}
=\frac{1+a}{1-b_a}>1.
\]
Thus both endpoints are repelling. For this Fibonacci parameter, Singer's theorem \cite[Theorem~II.6.1]{MS} therefore excludes attracting periodic orbits. There are also no wandering intervals.

\section{Semi-conjugacy}

Assume that $f\in\mathscr B^+$ is a symmetric Fibonacci bimodal map. The case when $f \in \mathscr B^-$ can be deduced to this case. For $k\geq1$, let
\[
U^0\supset U^1\supset U^2\supset\cdots\supset\{d\}, V^0\supset V^1\supset V^2\supset\cdots\supset\{c\}
\]
be its twin principal nest, and let $\phi_k$ denote the first return map to $U^{k-1}\cup V^{k-1}$. Recall that
\[
\phi_k=
\begin{cases}
f^{S_k},&\text{on }U^k\cup V^k,\\
f^{S_{k-1}},&\text{on }D^k\cup C^k.
\end{cases}
\]
Since $f$ is odd, the fixed point $p\in(d,c)$ is $p=0$. Moreover, by symmetry,
\[
U^k=-V^k, D^k=-C^k
\]
for every $k\geq1$.

Let
\[
I^0:=V^0, I^1:=V^1, J^1:=C^1,
\]
and define $g:I^1\cup J^1\to I^0$ by
\[
g=-\phi_1.
\]
More precisely,
\[
g=
\begin{cases}
-f^2,&\text{on }I^1,\\
-f,&\text{on }J^1.
\end{cases}
\]
By Fact~2.1, $\phi_1$ is of type $\mathcal A^{++}$. Hence $\phi_1(V^1)\subset U^0$ and $\phi_1(C^1)=U^0$, so that $g(I^1)\subset I^0$ and $g(J^1)=I^0$. Moreover, $g|_{I^1}$ has a unique local maximum at $c$, while $g|_{J^1}$ is monotone and onto. Since $f$ is odd,
\[
g(c)=-f^2(c)=f^2(d)\in C^1=J^1.
\]
Thus $g:I^1\cup J^1\to I^0$ is a box mapping with critical point $c$.

Write
\[
U^n=(u_n,\hat u_n), V^n=(\hat v_n,v_n),
\]
where $u_n<d<c<v_n$. By symmetry,
\[
u_n=-v_n, \hat u_n=-\hat v_n.
\]
In particular, $\hat u_0=\hat v_0=p=0$ and $u_0=-v_0$. Define
\[
\pi:U^0\cup V^0\to I^0,
\pi(x)=
\begin{cases}
-x,&x\in U^0,\\
x,&x\in V^0.
\end{cases}
\]
Then $\pi$ is piecewise linear and two-to-one. The following lemma shows that $g$ is semi-conjugate to the restriction of $\phi_1$ to $D^1\cup U^1\cup V^1\cup C^1$.

\begin{lem}
The following diagram commutes:
\[
\begin{CD}
D^1\cup U^1\cup V^1\cup C^1 @>{\phi_1}>> U^0\cup V^0\\
@V{\pi}VV @VV{\pi}V\\
I^1\cup J^1 @>{g}>> I^0.
\end{CD}
\]
Equivalently,
\[
\pi\circ\phi_1=g\circ\pi
\]
on $D^1\cup U^1\cup V^1\cup C^1$.
\end{lem}

\begin{proof}
For $x\in U^1$, we have $\pi(x)=-x\in V^1$ and $\phi_1(x)=f^2(x)\in V^0$. Since $f^2$ is odd,
\[
g(\pi(x))=-f^2(-x)=f^2(x)=\pi(\phi_1(x)).
\]
For $x\in V^1$, we have $\pi(x)=x$ and $\phi_1(x)=f^2(x)\in U^0$, and hence
\[
g(\pi(x))=-f^2(x)=\pi(\phi_1(x)).
\]
For $x\in D^1$, we have $\pi(x)=-x\in C^1$ and $\phi_1(x)=f(x)\in V^0$. Therefore
\[
g(\pi(x))=-f(-x)=f(x)=\pi(\phi_1(x)).
\]
Finally, for $x\in C^1$, we have $\pi(x)=x$ and $\phi_1(x)=f(x)\in U^0$, so
\[
g(\pi(x))=-f(x)=\pi(\phi_1(x)).
\]
This proves the lemma.
\end{proof}

For each $k\geq1$, set
\[
I^k:=V^k, J^k:=C^k.
\]
Then, by symmetry,
\[
\pi(U^k)=\pi(V^k)=I^k, \pi(D^k)=\pi(C^k)=J^k.
\]
In the following we shall show that $g$ is an infinitely Fibonacci renormalizable box mapping in the sense of Definition 2.2.

By Fact~2.1 and induction, for every $k\geq2$,
\[
\phi_k|_{U^k\cup V^k}=\phi_1^{S_{k-1}},
\phi_k|_{D^k\cup C^k}=\phi_1^{S_{k-2}}.
\]
Indeed, this is clear for $k=2$, and the induction step follows from
\[
\phi_k|_{U^k\cup V^k}=\phi_{k-1}^2,
\phi_k|_{D^k\cup C^k}=\phi_{k-1},
\]
together with $S_{k-1}=S_{k-2}+S_{k-3}$.

\begin{lem}
For every $k\geq1$, the following hold:
\begin{enumerate}
\item[(1)] $g^{S_k}(I^{k+1})\subset I^k$ and $g^{S_k}(\partial I^{k+1})\subset\partial I^k$;
\item[(2)] $g^{S_{k-1}}:J^{k+1}\to I^k$ is monotone and onto. In particular, $g^{S_{k-1}}(\partial J^{k+1})=\partial I^k$;
\item[(3)] $g^{S_k}|_{I^{k+1}}$ has a unique local extremum at $c$, with $g^{S_k}(c)\in J^{k+1}$. Moreover, $g^{S_k}(I^{k+1})\supset I^{k+1}$.
\end{enumerate}
\end{lem}

\begin{proof}
By the preceding observation,
\[
\phi_{k+1}|_{U^{k+1}\cup V^{k+1}}=\phi_1^{S_k},
\phi_{k+1}|_{D^{k+1}\cup C^{k+1}}=\phi_1^{S_{k-1}}.
\]
Iterating the semiconjugacy relation gives
\[
g^{S_k}\circ\pi=\pi\circ\phi_{k+1}
\]
on $U^{k+1}\cup V^{k+1}$. Hence
\[
g^{S_k}(I^{k+1})
=\pi\bigl(\phi_{k+1}(U^{k+1}\cup V^{k+1})\bigr)
\subset I^k.
\]
Moreover,
\[
g^{S_k}(\partial I^{k+1})
=\pi\bigl(\phi_{k+1}(\partial U^{k+1}\cup\partial V^{k+1})\bigr)
\subset\partial I^k.
\]
By Fact~2.1,
\[
\phi_{k+1}(U^{k+1}\cup V^{k+1})
\supset U^{k+1}\cup V^{k+1},
\]
and therefore $g^{S_k}(I^{k+1})
\supset I^{k+1}$.

Similarly,
\[
g^{S_{k-1}}\circ\pi=\pi\circ\phi_{k+1}
\]
on $D^{k+1}\cup C^{k+1}$. Since the restrictions of $\phi_{k+1}$ to $D^{k+1}$ and $C^{k+1}$ are monotone and map these intervals onto $U^k$ and $V^k$ in the order determined by the type, it follows that $g^{S_{k-1}}:J^{k+1}\to I^k$ is monotone and onto. In particular, $g^{S_{k-1}}(\partial J^{k+1})=\partial I^k$.

Finally, $\phi_{k+1}|_{U^{k+1}}$ and $\phi_{k+1}|_{V^{k+1}}$ have unique local extrema at $d$ and $c$, respectively. Since $\pi(d)=\pi(c)=c$, these two extrema project to the unique local extremum $c$ of $g^{S_k}|_{I^{k+1}}$. Furthermore,
\[
g^{S_k}(c)
=\pi\bigl(\phi_{k+1}(\{d,c\})\bigr)
\in\pi(D^{k+1}\cup C^{k+1})
=J^{k+1}.
\]
This proves the lemma.
\end{proof}

Denote $g_1=g$, and for each $k\geq2$ define
\[
g_k=
\begin{cases}
g^{S_{k-1}},&\text{on }I^k,\\
g^{S_{k-2}},&\text{on }J^k.
\end{cases}
\]

\begin{lem}
For each $k\geq1$, the following diagram commutes:
\[
\begin{CD}
D^k\cup U^k\cup V^k\cup C^k @>{\phi_k}>> U^{k-1}\cup V^{k-1}\\
@V{\pi}VV @VV{\pi}V\\
I^k\cup J^k @>{g_k}>> I^{k-1}.
\end{CD}
\]
Equivalently,
\[
\pi\circ\phi_k=g_k\circ\pi
\]
on $D^k\cup U^k\cup V^k\cup C^k$.
\end{lem}

\begin{proof}
For $k=1$, this is the previous semiconjugacy lemma. Let $k\geq2$. On $U^k\cup V^k$ we have $\phi_k=\phi_1^{S_{k-1}}$, and hence
\[
\pi\circ\phi_k
=\pi\circ\phi_1^{S_{k-1}}
=g^{S_{k-1}}\circ\pi
=g_k\circ\pi.
\]
Similarly, on $D^k\cup C^k$, $\phi_k=\phi_1^{S_{k-2}}$, so
\[
\pi\circ\phi_k
=\pi\circ\phi_1^{S_{k-2}}
=g^{S_{k-2}}\circ\pi
=g_k\circ\pi.
\]
Thus the diagram commutes.
\end{proof}

\begin{prop}
For every $k\geq1$, $g_k$ is Fibonacci renormalizable and
\[
\mathcal I g_k=g_{k+1}.
\]
Hence $g$ is infinitely Fibonacci renormalizable.
\end{prop}

\begin{proof}
We prove the statement by induction on $k$. For $k=1$, by the construction of $g$ and Fact~2.1, $g(I^1)\supset I^1$. Moreover, $g(c)\in J^1$, so $g(c)\notin I^1$. Since
\[
\phi_2|_{U^2\cup V^2}=\phi_1^2,
\phi_2|_{D^2\cup C^2}=\phi_1,
\]
the semiconjugacy gives $g^2(c)\in J^2\subset I^1$. Since the post-critical branches of $\phi_2$ map $D^2$ and $C^2$ onto $U^1$ and $V^1$ in the appropriate order, we also have $g^3(c)\in I^1$.  Thus
\[
g(I^1)\supset I^1,
g(c)\notin I^1,
g^2(c),g^3(c)\in I^1,
\]
and hence $g=g_1$ is Fibonacci renormalizable.

Furthermore, $I^2$ is the central domain of the first return map to $I^1$. Indeed, $I^2=\pi(U^2)=\pi(V^2)$, where $U^2$ and $V^2$ are the central return domains of $\phi_1$ to $U^1\cup V^1$, and
\[
\phi_2|_{U^2\cup V^2}=\phi_1^2.
\]
Similarly, since $g^2(c)\in J^2$ and
\[
J^2=\pi(D^2)=\pi(C^2),
\]
where $D^2$ and $C^2$ are the corresponding post-critical return domains, $J^2$ is the return domain containing $g^2(c)$. The relations
\[
\phi_2|_{U^2\cup V^2}=\phi_1^2,
\phi_2|_{D^2\cup C^2}=\phi_1
\]
and the semiconjugacy imply
\[
g_2|_{I^2}=g_1^2|_{I^2},
g_2|_{J^2}=g_1|_{J^2}.
\]
Therefore $\mathcal I g_1=g_2$. 

Now assume that, for some $k\geq1$, $g_k=\mathcal I^{k-1}g$. By Fact~2.1,
\[
\phi_{k+1}|_{U^{k+1}\cup V^{k+1}}=\phi_k^2,
\phi_{k+1}|_{D^{k+1}\cup C^{k+1}}=\phi_k.
\]
Using the commutative diagrams for $\phi_k$ and $\phi_{k+1}$, we obtain
\[
g_{k+1}|_{I^{k+1}}=g_k^2|_{I^{k+1}},
g_{k+1}|_{J^{k+1}}=g_k|_{J^{k+1}}.
\]
Moreover, $I^{k+1}$ is the central domain of the first return map to $I^k$, while $J^{k+1}$ is the return domain containing $g_k^2(c)$. Hence $g_k$ is Fibonacci renormalizable and
\[
\mathcal I g_k=g_{k+1}.
\]
Therefore
\[
g_{k+1}=\mathcal I^k g.
\]
The result follows by induction.
\end{proof}

\begin{proof}[Proof of the Main Theorem]
It suffices to consider $f\in\mathscr B^+$. Indeed, if $f\in\mathscr B^-$, put $\widehat f=-f$. Since $f$ is odd, $\widehat f^{\,n}(x)=(-1)^n f^n(x)$. Thus $\widehat f$ is a positive symmetric Fibonacci bimodal map with the same critical orders. By Fact~2.1 and symmetry, the two maps have the same critical $\omega$-limit set $A$, with $A=-A$. Hence $\operatorname{dist}(\widehat f^{\,n}(x),A)=\operatorname{dist}(f^n(x),A)$ for every $n\geq0$, so ${\rm Rel}_{\widehat f}(A)={\rm Rel}_f(A)$. The conclusion for $\widehat f$ therefore implies the conclusion for $f$.

We shall use the box mapping $g$ constructed above.
Since $f(c)$ is not a critical point, $f$ is a local diffeomorphism near $f(c)$. Thus $g=-f^2$ near $c$ has a non-flat critical point of order $\ell$. By the preceding proposition and the construction of $g$, the map $g$ satisfies the hypotheses of Theorem~1. Let $\ell_0$ be as in Theorem~1 and assume that $\ell\geq\ell_0$. Then
\[
A_g:=\omega_g(c)
\]
is a minimal Cantor set and
\[
B_g:={\rm Rel}_g(A_g)
\]
has positive Lebesgue measure. Set
\[
\widetilde B:=\pi^{-1}(B_g).
\]
Since $\pi$ is linear with slope $\pm1$ on $U^0$ and $V^0$, we have ${\rm Leb}(\widetilde B)>0$.

Let
\[
A_1:=\omega_{\phi_1}(c)=\omega_{\phi_1}(d).
\]
We claim that $\pi(A_1)=A_g$ and that $\pi|_{A_1}:A_1\to A_g$ is two-to-one. Since
\[
\pi\circ\phi_1=g\circ\pi, \pi(c)=c,
\]
we have $\pi(\phi_1^n(c))=g^n(c)$ for every $n\geq0$. Taking $\omega$-limits gives
\[
\pi(A_1)=A_g.
\]
Moreover, $\phi_1$ is odd and $d=-c$, so
\[
-A_1=-\omega_{\phi_1}(c)=\omega_{\phi_1}(d)=A_1.
\]
Since $0\notin A_1$ and $\pi$ identifies precisely the two points $x$ and $-x$ in $U^0\cup V^0$, it follows that $\pi|_{A_1}$ is two-to-one. In particular,
\[
\pi^{-1}(A_g)=A_1.
\]

Now let $y\in\widetilde B$. By the semiconjugacy,
\[
\pi\bigl(\omega_{\phi_1}(y)\bigr)\subset\omega_g(\pi(y))\subset A_g.
\]
Hence
\[
\omega_{\phi_1}(y)\subset\pi^{-1}(A_g)=A_1,
\]
and therefore
\[
\widetilde B\subset{\rm Rel}_{\phi_1}(A_1).
\]

Since the return time of $\phi_1$ is either $1$ or $2$, we have
\[
A:=\omega_f(c)=A_1\cup f(A_1).
\]
For every $y\in\widetilde B$,
\[
\omega_f(y)\subset\omega_{\phi_1}(y)\cup f\bigl(\omega_{\phi_1}(y)\bigr)
\subset A_1\cup f(A_1)=A.
\]
Hence
\[
\widetilde B\subset{\rm Rel}_f(A),
\]
and consequently
\[
{\rm Leb}({\rm Rel}_f(A))\geq{\rm Leb}(\widetilde B)>0.
\]
Since $A=\omega_f(c)=\omega_f(d)$ is a minimal Cantor set and is not solenoidal, $A$ is a wild Cantor attractor of $f$. This completes the proof.
\end{proof}


\begin{thebibliography}{aa}


\bibitem{A}
L. Alvin. {\it Uniformly recurrent sequences and minimal Cantor omega-limit sets.} Fund. Math. {\bf 231}(3) (2015), 273--284.

\bibitem{AC}
L. Alvin, J. \v{C}in\v{c}. {\it Recurrence, symbolic dynamics, and wild attractors for unimodal maps.} 	arXiv:2606.02308.

\bibitem{BL}
A. Blokh, M. Lyubich. {\it Attractors of maps of the interval.} Banach Center Publ. {\bf 23} (1989), 427--442.

\bibitem{BM}
A. Blokh, M. Misiurewicz. {\it Wild attractors of polymodal negative Schwarzian maps.} Comm. Math. Phys. {\bf 192}(2) (1998), 397--416.

\bibitem{BH}
B. Branner, J. H. Hubbard. {\it The iteration of cubic polynomials I.} Acta Math. {\bf 160} (1988), 143--206.

\bibitem{B}
H. Bruin. {\it Topological conditions for the existence of absorbing Cantor sets.} Trans. Amer. Math. Soc. {\bf 350} (1998), 2229--2263.

\bibitem{B1}
H. Bruin. {\it Combinatorics of the kneading map.} Internat. J. Bifur. Chaos Appl. Sci. Engrg. {\bf 5}(5) (1995), 1339--1349.

\bibitem{BKNS}
H. Bruin, G. Keller, T. Nowicki, S. van Strien. {\it Wild Cantor attractors exist.} Ann. of Math. {\bf 143} (1996), 97--130.

\bibitem{Bu}
X. Buff. {\it Fibonacci fixed point of renormalization.} Ergodic Theory Dynam. Systems {\bf 20} (2000), 1287--1317.


\bibitem{GvS}
R. Galeeva, S. van Strien. {\it Which families of $\ell$-modal maps are full?} Trans. Amer. Math. Soc. {\bf 348}(8) (1996), 3215--3221.

\bibitem{GSS}
J. Graczyk, D. Sands, G. \'Swi\c atek. {\it Metric attractors for smooth unimodal maps.} Ann. of Math. {\bf 159} (2004), 725--740.


\bibitem{HK}
F. Hofbauer, G. Keller. {\it Some remarks on recent results about S-unimodal maps.} Ann. Inst. H. Poincar\'e Phys. Th\'eor. {\bf 53} (1990), 413--425.


\bibitem{LSw}
G. Levin, G. \'Swi\c atek. {\it Limit drift.} Ergodic Theory Dynam. Systems {\bf 37}(8) (2017), 2643--2670.


\bibitem{LW}
S. Li, Q. Wang. {\it A new class of generalized fibonacci unimodal maps.} Nonlinearity {\bf 27}
(2014), 1633--1643.

\bibitem{LS}
S. Li, W. Shen. {\it On unimodal maps with critical order $2+\epsilon$.} Fund. Math. {\bf 192} (2006), 77--86.

\bibitem{JL}
H. Ji, S. Li. {\it The Attractor of Fibonacci-like Renormalization Operator.} Acta Math. Sin. (Engl. Ser.) {\bf 36} (2020), 1256--1278.

\bibitem{JL2}
H. Ji, S. Li. {\it On the Combinatorics of Fibonacci-Like Non-renormalizable Maps.} Commun. Math. Stat. {\bf 8} (2020), 473--496.

\bibitem{JM}
H. Ji, W. Ma. {\it Decay of geometry for a class of cubic polynomials.} Commun. Math. Stat. published online (2025). https://doi.org/10.1007/s40304-024-00413-6.

\bibitem{JM2}
H. Ji, W. Ma. {\it Phase transition for piecewise linear fibonacci bimodal map.} AIMS Math. {\bf 8}(4) (2023), 8403--8430.


\bibitem{KN}
G. Keller, T. Nowicki. {\it Fibonacci maps re(al)-visited.} Ergodic Theory Dynam. Systems {\bf 15} (1995), 99--120.

\bibitem{LM}
M. Lyubich, J. Milnor. {\it The Fibonacci unimodal map.} J. Amer. Math. Soc. {\bf 6}(2) (1993), 425--457.

\bibitem{Lyu}
M. Lyubich. {\it Combinatorics, geometry and attractors of quasi-quadratic maps.} Ann. of Math. {\bf
140} (1994), 347--404.

\bibitem{Mil}
J. Milnor. {\it On the concept of attractor.} Comm. Math. Phys. {\bf 99}(2) (1985), 177--195. Correction and remarks: Comm. Math. Phys. {\bf 102}(3) (1985), 517--519.
  
\bibitem{MS}
W. de Melo, S. van Strien. {\it  One-dimensional dynamics.}
Springer-Verlag, Berlin, 1993.

\bibitem{NvS}
T. Nowicki, S. van Strien. {\it Polynomial maps with a Julia set of positive measure.}  arXiv:math/9402215.

\bibitem{S}
W. Shen. {\it Decay of geometry for unimodal maps: an elementary proof.} Ann. of Math. {\bf 163}
(2006), 383--404.

\bibitem{Sm}
D. Smania. {\it Puzzle geometry and rigidity: The Fibonacci cycle is hyperbolic.} J. Amer. Math. Soc. {\bf 20}(3) (2007), 629--673.

\bibitem{SV}
G. \'Swi\c atek, E. Vargas. {\it Decay of geometry in the cubic family.} Ergodic Theory Dynam. Systems {\bf 18} (1998), 1311--1329.

\bibitem{vSV}
S. van Strien, E. Vargas. {\it Real bounds, ergodicity and negative Schwarzian for multimodal maps.} J. Amer. Math. Soc. {\bf 17}(4) (2004), 749--782. Erratum: J. Amer. Math. Soc. {\bf 20}(1) (2007), 267--268.

\bibitem{V}
E. Vargas. {\it Fibonacci bimodal maps.} Discrete Contin. Dyn. Syst. {\bf 22}(3) (2008), 807--815.

\bibitem{VY}
J. Olivares-Vinales, S. Yoo. {\it Towers and Bratteli-Vershik systems in Fibonacci-like unimodal maps.} 	arXiv:2602.21623.

\bibitem{Z}
R. Zhang. {\it New Fibonacci-like wild attractors for unimodal interval maps.} Ph.D. thesis, National University of Singapore, 2015.


\end{thebibliography}
\end{document}